\documentclass[a4paper,12pt]{amsart}
\usepackage{lipsum}
\usepackage{amssymb,amsthm}
\usepackage[foot]{amsaddr}
\usepackage{amsmath}
\usepackage{tikz}
\usetikzlibrary{quotes,angles}
\usepackage{graphicx}
\usepackage{subcaption}
\usepackage[shortlabels]{enumitem}
\usetikzlibrary{arrows}
\usepackage{float}
\usepackage{graphicx}
\usepackage{subcaption}
\usepackage{hyperref}
\usepackage[section]{placeins}
\graphicspath{{non_linear_Compton_figs/}}
\usetikzlibrary{decorations.pathmorphing}
\tikzset{snake it/.style={decorate, decoration=snake}}
\usepackage{pgfplots}
\pgfplotsset{compat=1.18}
\makeatletter
\newcommand*{\rom}[1]{\expandafter\@slowromancap\romannumeral #1@}
\makeatother

\numberwithin{equation}{section}

\theoremstyle{plain}
\newtheorem{theorem}{Theorem}
\numberwithin{theorem}{section}

\newtheorem{lemma}[theorem]{Lemma}
\newtheorem{corollary}[theorem]{Corollary}
\theoremstyle{definition}
\newtheorem{definition}[theorem]{Definition}
\theoremstyle{remark}
\newtheorem{remark}[theorem]{Remark}

\theoremstyle{remark}

\theoremstyle{remark}

\newcommand{\smo}{\setminus \mathbf{0}}

\newcommand{\norm}[1]{\left\lVert#1\right\rVert}      % Norm
\newcommand{\abs}[1]{\left|#1\right|}                 % Absolutbetrag
\newcommand{\paren}[1]{\left(#1\right)}               % Klammern
\newcommand{\sparen}[1]{\left\{#1\right\}}      % Mengenklammer
\renewcommand{\d}{\,\mathrm{d}}  % within the integral sign

\newcommand{\dd}{\mathrm{d}}  % without the space
\newcommand{\Cc}{\mathcal{C}}

\newcommand{\Dc}{\mathcal{D}}
\newcommand{\Ec}{\mathcal{E}}
\newcommand{\Fc}{\mathcal{F}}

\newcommand{\Sc}{\mathcal{S}}

\newcommand{\WF}{\mathrm{WF}}                         % Wavefront set
\newcommand{\wf}{\mathrm{WF}}                         % Wavefront set

\newcommand{\vb}{\mathbf{b}}

\newcommand{\partyf}[2]{\frac{\partial #2}{\partial y_{#1}}}

\newcommand{\vv}{{\mathbf{v}}}

\newcommand{\bpm}{\begin{pmatrix}}
\newcommand{\epm}{\end{pmatrix}}

\newcommand{\vx}{{\mathbf{x}}}
\newcommand{\vxo}{\mathbf{x}_0}

\newcommand{\vy}{{\mathbf{y}}}

\newcommand{\vs}{\mathbf{s}}
\newcommand{\vd}{\mathbf{d}}

\newcommand{\vxi}{{\boldsymbol{\xi}}}
\newcommand{\vxio}{{\boldsymbol{\xi}_0}}

\newcommand{\vsig}{{\boldsymbol{\sigma}}} %when $\sigma$ is one dimensional, 
\newcommand{\rr}{{{\mathbb R}}}

\newcommand{\rn}{{{\mathbb R}^n}}

\newcommand{\drn}{\dot{\mathbb{R}^n}}

\newcommand{\be}{\begin{equation}}
\newcommand{\bea}{\begin{eqnarray}}
\newcommand{\eea}{\end{eqnarray}}
\newcommand{\bean}{\begin{eqnarray*}}
\newcommand{\eean}{\end{eqnarray*}}

\newcommand{\bel}[1]{\begin{equation}\label{#1}}
\newcommand{\ee}{\end{equation}}
\newcommand{\eel}[1]{{\label{#1}\end{equation}}}

\DeclareMathOperator*{\argmin}{arg\,min}
\DeclareMathOperator{\sinc}{sinc}

\usepackage{fullpage}

\usepackage{kbordermatrix}
\renewcommand{\kbldelim}{(}% Left delimiter
\renewcommand{\kbrdelim}{)}% Right delimiter

\usepackage{datetime}
\usetikzlibrary{quotes, angles}

\title[short]{Microlocal analysis of non-linear operators arising in Compton CT\\{\footnotesize\ddmmyyyydate\today~\currenttime}}
\author{James W. Webber\textsuperscript{$\dagger$}}
\author{Sean Holman\textsuperscript{$\ddagger$}}
\address[James W. Webber (corresponding author)]{Department of Biomedical Engineering, Cleveland Clinic, USA}
\address[Sean Holman]{Department of Mathematics, The University of Manchester, Alan Turing Building, Oxford Road, Manchester M13 9PY, UK}
\email[A1,A2]{webberj5@ccf.org\textsuperscript{$\dagger$}, sean.holman@manchester.ac.uk\textsuperscript{$\ddagger$}}
\providecommand{\keywords}[1]
{
  \small	
  \textbf{\textit{Keywords---}} #1
}

\begin{document}

\begin{abstract}
We present a novel microlocal analysis of a non-linear ray transform, $\mathcal{R}$, arising in Compton Scattering Tomography (CST). Due to attenuation effects in CST, the integral weights depend on the reconstruction target, $f$, which has singularities. Thus, standard linear Fourier Integral Operator (FIO) theory does not apply as the weights are non-smooth. The V-line (or broken ray) transform, $\mathcal{V}$, can be used to model the attenuation of incoming and outgoing rays. Through novel analysis of $\mathcal{V}$, we characterize the location and strength of the singularities of the ray transform weights. In conjunction, we provide new results which quantify the strength of the singularities of distributional products based on the Sobolev order of the individual components. By combining this new theory, our analysis of $\mathcal{V}$, and classical linear FIO theory, we determine the Sobolev order of the singularities of $\mathcal{R}f$. The strongest (lowest Sobolev order) singularities of $\mathcal{R}f$ are shown to correspond to the wavefront set elements of the classical Radon transform applied to $f$, and we use this idea and known results on the Radon transform to prove injectivity results for $\mathcal{R}$. In addition, we present novel reconstruction methods based on our theory, and we validate our results using simulated image reconstructions.
\end{abstract}
\maketitle

\keywords{{\it{\textbf{Keywords}}}} - microlocal analysis, Sobolev estimates, distributional products

\section{Introduction}
In this paper, we present a novel analysis of a non-linear Radon transform, $\mathcal{R}$, which models the intesity of photons in Compton Scattering Tomography (CST). Specifically, we investigate how $\mathcal{R}$ propagates singularities of $f$, and we compare this to the classical linear Radon transform, $R$. In terms of Sobolev scale, the most irregular (i.e., lowest Sobolev order) singularities of $\mathcal{R}f$ are shown to have a 1-1 correspondance with the singularities of $Rf$. We use this result to recover the singularities of $f$ and derive injectivity results for $\mathcal{R}$ when $f$ is a characteristic function supported on a simply connected domain. Through simulated experiments, we validate our theory and also discuss some of the weaknesses of our method that occur in practice.

The literature considers several scanning modalities in CST \cite{norton,NT,pal,me2,rigaud20183d,rigaud20213d,rigaud2021reconstruction,farmer1971new,farmer1974further,holt1985compton,strecker1982scatter}.  A review of such works in provided in \cite{truong2012recent}. The majority of these works in recent years focus on the case when the gamma ray source and detector are uncollimated. The mathematical model for this data is a circular arc transform in 2-D \cite{norton,pal}, and a spindle torus transform in three-dimensions \cite{me2,rigaud20183d}. For example, in \cite{norton}, the authors present a Filtered Backprojection (FBP) type formula for a linear circular arc transform, $\mathcal{I}$. This transform was shown earlier in \cite{Co1963} to be invertible using Fourier series decomposition. However, the authors in \cite{norton} argue that their closed-form solution has better practical benefit as it is easier to implement and less sensitive to noise.

In \cite{NT,pal}, a different variant of circular Radon transform in considered, which we denote $C$. In \cite{NT}, the authors provide an inversion formula for $C$ using Fourier decomposition ideas and formulae provided by Cormack in \cite{Co1963}. An FBP formula for $C$, similar to the classical FBP formula for $R$, is derived in \cite{pal}. Later, in \cite{webberholman,truong2019compton}, $C$ was shown to be closely related to $R$ via a projection map. The authors in \cite{truong2019compton} combine this idea with algorithms developed for the exterior Radon transform \cite{quinto1988tomographic} to derive an inversion method for $C$ on annular domains.

In \cite{farmer1971new,farmer1974further,holt1985compton,strecker1982scatter}, various linear CST modalities are considered, where the photons are contrained to scatter on planes. In this case, the sources are typically collimated to pencil beams. For example, in \cite{farmer1971new}, the authors pair a pencil beam source with a wide-angle collimator on the detector to rapidly measure the electron density along a line profile. Then, multiple parallel lines are pieced together to form a 2-D image. In this case, given the extensive collimation, the density can be measured directly, and no inversion is required (the problem is well-posed). However, the amount of noise and scan time is likely amplified.

In the papers discussed so far, the inverse problem is linear as the effects of attenuation are omitted from the model. This assumption is not ideal, especially for larger and/or more dense material, and can cause artifacts in the reconstruction. We introduce a novel non-linear CST system which addresses attenuation effects, and we provide a rigorous analysis of the singularities of the data, and their order on Sobolev scale.

In \cite{rigaud20183d}, a 3-D CST scanner is proposed whereby $f$ is recovered from its integrals over spindle tori. Here, attenuation is addressed, and the authors derive methods to recover the edges of $f$ using classical microlocal analysis. In order for this to work, the integral weights have to be smooth, which cannot be the case in the non-linear formulation as the weights depend on $f$ and $f$ is non-smooth. To deal with this, the authors convolve $f$ with a smooth kernel. This removes any singularities and one can no longer consider edge recovery. In practice, since the kernel has small radius, the smoothed $f$ exhibits sharp changes in value, and the authors are able to recover the locations where these occur (i.e., an approximation to $\text{ssupp}(f)$) using FBP methods. Our analysis does not require the weights to be smooth in this way, and we derive new theory to address the non-linearity. For our uniqueness results to hold, we convolve with a kernel in a similar vein to \cite{rigaud20183d} and this is designed to model the finite size of the source and detector (i.e., they are not point sources in practice). Our kernel need not be smooth, however (e.g., we consider characteristic functions). This relaxes significantly the assumptions of \cite{rigaud20183d}.

In \cite{rigaud20213d}, the authors consider multiple scattering effects in CST. Specifically, using the theory of linear FIO, it is proven that the higher order (i.e., 2nd order) scatter is smoother on Sobolev scale than the first order scatter. Thus, when the 1st and 2nd order scatter are added together, the first order singularities are most pronounced in the sinogram. Using this theory, and FBP methods, the authors show that the edges of the density can be recovered from the CST data. As in \cite{rigaud20183d}, the authors rely on linear FIO theory to prove their theorems. We consider a non-linear formulation where the weights are non-smooth and depend on $f$.

The remainder of this paper is organized as follows. In section \ref{sect:defns} , we recall some definitions and theorems from the literature that will be needed to prove our results. In section \ref{model}, we present our geometry and physical model and explain how the non-linearity arises. In section \ref{main_thm_section}, we present our main microlocal theory where we address the singularities of $\mathcal{R}f$ and their order on Sobolev scale. We then use this result to prove uniqueness results for $\mathcal{R}$ when $f$ is a characteristic function. In section \ref{simulations}, we present simulated experiments to validate our theory and we show edge map reconstructions of image phantoms.

\section{Definitions}\label{sect:defns} 

In this section, we review some theory from microlocal analysis which will be used in our theorems. We first provide some
notation and definitions.  Let $X$ and $Y$ be open subsets of
{$\mathbb{R}^{n_X}$ and $\mathbb{R}^{n_Y}$, respectively.}  Let $\Dc(X)$ be the space of smooth functions compactly
supported on $X$ with the standard topology and let $\mathcal{D}'(X)$
denote its dual space, the vector space of distributions on $X$.  Let
$\Ec(X)$ be the space of all smooth functions on $X$ with the standard
topology and let $\mathcal{E}'(X)$ denote its dual space, the vector
space of distributions with compact support contained in $X$. Finally,
let $\Sc(\rn)$ be the space of Schwartz functions, that are rapidly
decreasing at $\infty$ along with all derivatives. See \cite{Rudin:FA}
for more information. 

%If $A\subset \rn$, then $\intt(A)$ and $\partial A$ are, respectively, the interior of $A$ and the boundary of $A$.

We now list some notation conventions that will be used throughout this paper:
\begin{enumerate}
\item For a function $f$ in the Schwartz space $\Sc(\mathbb{R}^{n_X})$, we {write
\[
\mathcal{F}f(\xi) = \int_{\mathbb{R}^{n_X}} e^{-i x\cdot \xi}f(x)\ \dd x,\quad \mathcal{F}^{-1}f(x) = \frac{1}{(2 \pi)^{n_x}}\int_{\mathbb{R}^{n_X}} e^{i x\cdot \xi}f(\xi)\ \dd \xi
\]
for} the Fourier transform and inverse Fourier transform of $f$,
respectively, {and extend these in the usual way to tempered distributions $\Sc'(\mathbb{R}^{n_x})$} (see \cite[Definition 7.1.1]{hormanderI}). We have use the notation $\hat{f} = \mathcal{F}f$. %{$\mathcal{F}f$ and $\mathcal{F}^{-1}f$ are defined in terms of angular frequency.}
%Note that
%$$\Fc\inv \Fc f(\vx)= \frac{1}{(2\pi)^{n_X}}\int_{\vy\in\mathbb{R}^{n_X}}\int_{\vz\in
%\mathbb{R}^{n_X}} \exp((\vx-\vz)\cdot \vy)\,
%f(\vz)\d \vz\d \vy.$$

\item We use the standard multi-index notation: if
$\alpha=(\alpha_1,\alpha_2,\dots,\alpha_n)\in \sparen{0,1,2,\dots}^{n_X}$
is a multi-index and $f$ is a function on $\mathbb{R}^{n_X}$, then
\[\partial^\alpha f=\paren{\frac{\partial}{\partial
x_1}}^{\alpha_1}\paren{\frac{\partial}{\partial
x_2}}^{\alpha_2}\cdots\paren{\frac{\partial}{\partial x_{n_X}}}^{\alpha_{n_X}}
f.\] If $f$ is a function of $(\vy,\vx,\vsig)$ then $\partial^\alpha_\vy
f$ and $\partial^\alpha_\vsig f$ are defined similarly.

\item \label{item:T*Xident} We identify the cotangent
spaces of Euclidean spaces with the underlying Euclidean spaces. For example, the cotangent space, 
$T^*(X)$, of $X$ is identified with $X\times \mathbb{R}^{n_X}$. If $\Phi$ is a function of $(\vy,\vx,\vsig)\in Y\times X\times \rr^N$,
then we define $\dd_{\vy} \Phi = \paren{\partyf{1}{\Phi},
\partyf{2}{\Phi}, \cdots, \partyf{{n_X}}{\Phi} }$, and $\dd_\vx\Phi$ and $
\dd_{\vsig} \Phi $ are defined similarly. Identifying the cotangent space with the Euclidean space as mentioned above, we let $\dd\Phi =
\paren{\dd_{\vy} \Phi, \dd_{\vx} \Phi,\dd_{\vsig} \Phi}$.

%\item For $\Omega\subset \rr^m$, we define $\dot{\Omega}
%= \Omega\smo$.

\end{enumerate}

\noindent The singularities of a function and the directions in which they occur
are described by the wavefront set \cite[page
16]{duistermaat1996fourier}, which we now define.
\begin{definition}
\label{WF} Let $X$ be an open subset of $\mathbb{R}^{n_X}$ and let $f$ be a
distribution in $\mathcal{D}'(X)$.  Let $(\vx_0,\vxi_0)\in X\times
\drn$.  Then $f$ is \emph{smooth at $\vx_0$ in direction $\vxio$} if
there exists a neighborhood $U$ of $\vx_0$ and $V$ of $\vxi_0$ such
that for every $\Phi\in \Dc(U)$ and $N\in\mathbb{R}$ there exists a
constant $C_N$ such that for all $\vxi\in V$ {and $\lambda >1$},
\begin{equation}
\left|\Fc(\Phi f)(\lambda\vxi)\right|\leq C_N(1+\abs{\lambda})^{-N}.
\end{equation}
The pair $(\vx_0,\vxio)$ is in the \emph{wavefront set,} $\wf(f)$, if
$f$ is not smooth at $\vx_0$ in direction $\vxio$. We define for fixed $\vx_0$, $\WF_{\vx_0}(f) = \{\xi : (\vx_0,\xi) \in \WF(f)\}$, and we define the singular support of $f$, $\text{ssupp}(f)$, as the natural projection of $\WF(f)$ onto $X$.
\end{definition}

Intuitively, the elements 
$(\vxo,\vxio)\in \WF(f)$ are the point-normal vector pairs at
which $f$ has singularities; $\vxo$ is the location of the 
  singularity, and  $\vxio$ is the direction in which the
  singularity occurs.
  A
 geometric example of the wavefront set is given by the characteristic function $f$ of a domain $\Omega \subset \mathbb{R}^{n_X}$ with smooth boundary, which is $1$ on $\Omega$ and $0$ off of $\Omega$. Then the wavefront set is
\[
\WF(f) = \{(\vx,t\vv) \ : t \neq 0, \ \vx \in \partial \Omega, \ \mbox{$\vv$ is orthogonal to $\partial \Omega$ at $\vx$}\}.
\]
In other words, the wavefront set is the set of points in the boundary of $\Omega$ together with the nonzero normal vectors to the boundary. {The set of normals of the surface $\partial \Omega$ is a subset of the cotangent bundle $T^*X$, and here we are using the identification of $T^*X$ with $X \times \mathbb{R}^{n_X}$ mentioned above in point (\ref{item:T*Xident}).} The wavefront set is an important consideration in imaging since elements of the wavefront set will correspond to sharp features of an image.

%For example, if $f$ is the
%characteristic function on the unit ball, $B_n = \{\vx\in\mathbb{R}^n : |\vx|\leq 1\}$, in $\mathbb{R}^n$, then its
%wavefront set is $\WF(f)=\{(\vx,t\vx): \vx\in S^{n-1}, t\neq 0\}$, i.e., the
%set of points on $S^{n-1}$ (i.e., the boundary of $B$) paired with the corresponding normal vectors
%to $S^{n-1}$.

%\begin{equation}

%\end{equation}
%That is, 

The wavefront set of a distribution on $X$ is normally defined as a
subset the cotangent bundle $T^*(X)$ so it is invariant under
diffeomorphisms, but we do not need this invariance, so we will
continue to identify $T^*(X) = X \times \rr^{n_X}$ and consider $\WF(f)$ as
a subset of $X\times (\rr^{n_X}\smo)$.

%Let $X$ and $Y$ be open subsets of $\rn$, $m \in\mathbb{R}$.

 \begin{definition}[{\cite[Definition 7.8.1]{hormanderI}}] \label{ellip}We define
 $S^m(Y \times X, \mathbb{R}^N)$ to be the
set of $a\in \Ec(Y\times X\times \mathbb{R}^N)$ such that for every
compact set $K\subset Y\times X$ and all multi--indices $\alpha,
\beta, \gamma$ the bound
\[
\left|\partial^{\gamma}_{\vy}\partial^{\beta}_{\vx}\partial^{\alpha}_{\vsig}a(\vy,\vx,\vsig)\right|\leq
C_{K,\alpha,\beta,\gamma}(1+\norm{\vsig})^{m-|\alpha|},\ \ \ (\vy,\vx)\in K,\
\vsig\in\mathbb{R}^N,
\]
holds for some constant $C_{K,\alpha,\beta,\gamma}>0$. 

 The elements of $S^m$ are called \emph{symbols} of order $m$.  Note
that this symbol class is  sometimes denoted $S^m_{1,0}$.  The symbol
$a\in S^m(Y \times X,\rr^N)$ is \emph{elliptic} if for each compact set
$K\subset Y\times X$, there is a $C_K>0$ and $M>0$ such that
\bel{def:elliptic} \abs{a(\vy,\vx,\vsig)}\geq C_K(1+\norm{\vsig})^m,\
\ \ (\vy,\vx)\in K,\ \norm{\vsig}\geq M.
\ee 
\end{definition}

\begin{definition}[{\cite[Definition
        21.2.15]{hormanderIII}}] \label{phasedef}
A function $\Phi=\Phi(\vy,\vx,\vsig)\in
\Ec(Y\times X\times(\mathbb{R}^N\smo))$ is a \emph{phase
function} if $\Phi(\vy,\vx,\lambda\vsig)=\lambda\Phi(\vy,\vx,\vsig)$, $\forall
\lambda>0$ and $\mathrm{d}\Phi$ is nowhere zero. The
\emph{critical set of $\Phi$} is
\[\Sigma_\Phi=\{(\vy,\vx,\vsig)\in Y\times X\times(\mathbb{R}^N\smo)
: \dd_{\vsig}\Phi=0\}.\] 
 A phase function is
\emph{clean} if the critical set $\Sigma_\Phi$ is a smooth manifold {with tangent space defined {by} the kernel of $\mathrm{d}\,(\mathrm{d}_\sigma\Phi)$ on $\Sigma_\Phi$. Here, the derivative $\mathrm{d}$ is applied component-wise to the vector-valued function $\mathrm{d}_\sigma\Phi$. So, $\mathrm{d}\,(\mathrm{d}_\sigma\Phi)$ is treated as a Jacobian matrix of dimensions $N\times ({n_Y + n_X}+N)$.}
\end{definition}
\noindent By the {Constant Rank Theorem} \cite[Theorem 4.12]{lee2012smooth} the requirement for a phase
function to be clean is satisfied if
$\mathrm{d}\paren{\mathrm{d}_\vsig
\Phi}$ has constant rank.

\begin{definition}[{\cite[Definition 21.2.15]{hormanderIII} and
      \cite[section 25.2]{hormander}}]\label{def:canon} Let $X$ and
$Y$ be open subsets of $\rn$. Let $\Phi\in \Ec\paren{Y \times X \times
{\rr}^N}$ be a clean phase function.  In addition, we assume that
$\Phi$ is \emph{nondegenerate} in the following sense:
\[\text{$\dd_{\vy}\Phi$ and $\dd_{\vx}\Phi$ are never zero on
$\Sigma_{\Phi}$.}\]
  The
\emph{canonical relation parametrized by $\Phi$} is defined as
\begin{equation}\label{def:Cgenl} \begin{aligned} \Cc=&\sparen{
\paren{\paren{\vy,\dd_{\vy}\Phi(\vy,\vx,\vsig)};\paren{\vx,-\dd_{\vx}\Phi(\vy,\vx,\vsig)}}:(\vy,\vx,\vsig)\in
\Sigma_{\Phi}}{.}
% &\hspace{1.5cm} \vs\in \rr^N\smo,   
\end{aligned}
\end{equation}
\end{definition}

\begin{definition}\label{FIOdef}
Let $X$ and $Y$ be open subsets of {$\mathbb{R}^{n_X}$ and $\mathbb{R}^{n_Y}$, respectively.} {Let an operator $A :
\Dc(X)\to \mathcal{D}'(Y)$ be defined by the distribution kernel
$K_A\in \mathcal{D}'(Y\times X)$, in the sense that
$Af(\vy)=\int_{X}K_A(\vy,\vx)f(\vx)\mathrm{d}\vx$. Then we call $K_A$
the \emph{Schwartz kernel} of $A$}. A \emph{Fourier
integral operator (FIO)} of order $\mu = m + N/2 - (n_X+n_Y)/4$ is an operator
$A:\Dc(X)\to \mathcal{D}'(Y)$ with Schwartz kernel given by an
oscillatory integral of the form
\begin{equation} \label{oscint}
K_A(\vy,\vx)=\int_{\mathbb{R}^N}
e^{i\Phi(\vy,\vx,\vsig)}a(\vy,\vx,\vsig) \mathrm{d}\vsig,
\end{equation}
where $\Phi$ is a clean nondegenerate phase function and $a$ is a
symbol in $S^m(Y \times X , \mathbb{R}^N)$. The \emph{canonical
relation of $A$} is the canonical relation $\mathcal{C}$ of $\Phi$ defined in
\eqref{def:Cgenl}. $A$ is called an \emph{elliptic} FIO if its symbol is elliptic. An FIO is called a \emph{pseudodifferential operator} if {$X = Y$ and} its canonical relation $\Cc$ is contained in the diagonal, i.e.,
$\Cc \subset \Delta := \{ (\vx,\vxi;\vx,\vxi)\}$.
\end{definition}

%When $n_X = 0$ in Definition \ref{FIOdef}, \eqref{oscint} defines a distribution on $Y$ known as a \emph{Lagrangian distribution} corresponding to the Lagrangian manifold
%\[
%\Lambda = \{ (\vy, \dd_\vy \Phi(\vy,\vsig)) : (\vy,\vsig) \in \Sigma_\Phi\}.
%\]
%The set of Lagrangian distributions of order $\mu$ on $Y$ corresponding to a given Lagrange manifold $\Lambda$ will be denoted $I^\mu(Y,\Lambda)$. When $\Lambda = N^*(M)\setminus \{0\}$, the conormal bundle of some submanifold $M \subset Y$, we say the distributions is a \emph{conormal distribution}. For more detailed information about Lagrangian and conormal distributions, see \cite{hormanderIII,hormander}.

Fourier integral operators are defined in \cite{hormander} more generally as operators with Schwartz kernel locally given by expressions of the form \eqref{oscint} where the phase functions each give rise to pieces of the same global cannonical relation. However, the local expression \eqref{oscint} is sufficient for our purposes.

%\seanc{Material above about Lagrangian distributions may be able to be removed.}

We have the definition of Sobolev spaces from \cite[page 200]{natterer}.

\begin{definition}\label{sobo_spaces}
We define the Sobolev space order $\alpha$,
\begin{equation}
H^\alpha(\mathbb{R}^n) = \{f\in \mathcal{S}'(\mathbb{R}^n) : (1 + |\xi|^2)^{\alpha/2} \hat{f} \in L^2(\mathbb{R}^n)\}
\end{equation}
with norm 
\begin{equation}\label{HRnnorm}
\|f\|_{H^{\alpha}(\mathbb{R}^n)} = \|(1 + |\xi|^2)^{\alpha/2} \hat{f}\|_{L^2(\mathbb{R}^n)}.
\end{equation}
For $\Omega \subset \mathbb{R}^n$ and $\alpha = s + \sigma$ where $0\leq\sigma<1$ and $s \geq 0$ an integer, we define
\begin{equation} \label{HOrnom}
\begin{split}
\|f\|^2_{H^{\alpha}(\Omega)} &= \|f\|^2_{H^s(\Omega)} + \sum_{|k|= s}\int_{\Omega}\int_{\Omega} \frac{|\partial^kf(\vx) - \partial^kf(\vy)|^2}{|\vx - \vy|^{n+2\sigma}}\mathrm{d}\vx\mathrm{d}\vy,
\end{split}
\end{equation}
where $\|f\|^2_{H^s(\Omega)} = \sum_{|k|\leq s} \|\partial^kf\|_{L^2(\Omega)}$. We will also use the notation
\[
|f|^2_{H^{\alpha}(\Omega)} = \sum_{|k|= s}\int_{\Omega}\int_{\Omega} \frac{|\partial^kf(\vx) - \partial^kf(\vy)|^2}{|\vx - \vy|^{n+2\sigma}}\mathrm{d}\vx\mathrm{d}\vy
\]
for the semi-norm which appears on the left side of \eqref{HOrnom}. Further, we define $H^\alpha_c(\Omega) = \{f \in H^\alpha(\Omega) : \text{supp}(f)\ \text{compact, and}\ \text{supp}(f) \subset {\Omega}\}$.
\end{definition}

There is potential confusion in the above definitions when $\Omega = \mathbb{R}^n$ and $\alpha \geq 0$, but in fact the two norms \eqref{HRnnorm} and \eqref{HOrnom} are equivalent in that case.

%\seanc{Comment on equivalence of these norms for functions in $H^\alpha_c(\Omega)$?}

From \cite{sobolev}, we have the theorem which addresses Sobolev regularity through smooth compositions.

%\seanc{Fix citation above.}
%\seanc{I don't think this theorem is true as written. For example when $\gamma = 1$ it fails. I think there needs to be more hypotheses on $\gamma$ or maybe $u$. Or maybe it should be $\gamma \circ u \in H^{\alpha}_{loc}(\mathbb{R}^n)$.}
\begin{theorem}
\label{sobo}
Let $\Omega$ be a compact subset of $\mathbb{R}^n$ with smooth boundary. Let $\alpha$ be rational, and let $\gamma \in C^{\infty}(\mathbb{R})$ be such that $\gamma^{(j)}\in L^{\infty}(\mathbb{R})$ for any $j\geq 0$. Let $u \in H^{\alpha}(\Omega)$. Then, $\gamma \circ u \in H^{\alpha}(\Omega)$. 
\end{theorem}

\noindent We have the theorem regarding trace operators.
\begin{theorem}
\label{trace_thm}
Let $Tf(\vx') = f(\vx',0)$ be a trace operator. Then, for $\alpha > 1/2$,
\begin{equation}
T: H^{\alpha,2}(\mathbb{R}^n) \to H^{\alpha-1/2,2}(\mathbb{R}^{n-1})
\end{equation}
is a bounded map.
\end{theorem}

\noindent We also extend the definition of Sobolev spaces in the next definition.

\begin{definition}\label{gen_sobo}
    Let $H^{\alpha,p}(\mathbb{R}^n)$ be the Banach space defined by the norm
    \[
    \|f\|_{H^{\alpha,p}(\mathbb{R}^n)} = \| (1+|\xi|^2)^{\frac{\alpha}{2}} \hat{f}\|_p.
    \]
\end{definition}

\noindent We now define local and microlocal Sobolev regularity and wavefront sets \cite{Q1993sing}.

\begin{definition} \label{def:SobolevWF}
A distribution $g$ is in $H^{\alpha,p}$ locally near a point $\vx_0$ if and only if there exists a cut-off function $\varphi \in C_c^{\infty}(\mathbb{R}^n)$ with $\varphi(\vx_0) \neq 0$ such that $\varphi g \in H^{\alpha,p}(\mathbb{R}^n)$. The distribution $g$ is in $H^{\alpha,p}$ microlocally near $(\vx_0,\xi_0)$ if and only if there is a cut-off function $\varphi \in C^{\infty}_c(\mathbb{R}^n)$ with $\varphi(\vx_0) \neq 0$ and function $u(\xi)$ homogeneous of degree zero and smooth on $\mathbb{R}^n\backslash \{0\}$ with $u(\xi_0)\neq 0$ such that $(1+|\xi|^2)^{\alpha/2}u(\xi)\mathcal{F}(\varphi g)(\xi) \in L^p(\mathbb{R}^n)$. We define the Sobolev wavefront set
\begin{equation}
\WF^{\alpha,p}(g) = \left\{ (\vx,\xi) \in \WF(g) : g\ \text{is in}\ H^{\alpha,p}\ \text{microlocally near}\ (\vx, \xi) \right\}
\end{equation}
and $\text{ssupp}^{\alpha,p}(g)$ as the natural projection of $\WF^{\alpha,p}(g)$ onto $X$. If $p$ is omitted in any of the cases above, then we assume the standard Sobolev space with $p = 2$.
\end{definition}

% \noindent We also extend naturally the definition of Sobolev spaces below

% \begin{definition}\label{gen_sobo}
%     Let $H^{\alpha,p}(\mathbb{R}^n)$ be the Banach space defined by the norm
%     \[
%     \|f\|_{H^{\alpha,p}(\mathbb{R}^n)} = \| \hat{f} (1+|\xi|^2)^{\frac{\alpha}{2}}\|_p.
%     \]
% \end{definition}

% For the generalized Sobolev spaces above, we define local Sobolev regularity similarly to Definition \ref{def:SobolevWF}, which applies to the special case when $p = 2$. 

\noindent The weighted V-line transform \cite{mishra2025tensor} is defined below.

\begin{definition}\label{vline_def}
Let $f \in L^2_c(\mathbb{R}^2)$. Then, for $a,b\in\mathbb{R}$, we define the weighted V-line transform with opening angle $\psi$ by
\begin{equation}
\label{vline}
\begin{split}
\mathcal{V}_{a,b}f(\vx,\phi) &= \mathcal{V}_{a,b}f(\vx, \Phi) \\
&= a\int_0^\infty f(\vx + t \Phi) \mathrm{d}t + b \int_0^\infty f( \vx + t \Phi' ) \mathrm{d}t\\
&= aLf(\vx,\Phi) +b Lf(\vx,\Phi'),
\end{split}
\end{equation}
where $\Phi = (\cos(\phi + 2\psi - \pi/2),\sin(\phi + 2\psi - \pi/2))$ and $\Phi' = (\cos(\phi - \pi/2),\sin(\phi - \pi/2))$, and $2\psi \in (0,\pi)$ is the angle between the V-lines. Here, $Lf(\vx,\Phi)$ is a divergent beam transform \cite{hamaker1980divergent} which defines the integrals of $f$ over the divergent beam $\{\vx + t\Phi : t \geq0 \}$. As \eqref{vline} would suggest, for convenience of notation, we interchange the angle $\phi \in \mathbb{R}$ and the vector $\Phi\in S^1$ as input variables for $\mathcal{V}_{a,b}f$ and $Lf$.

We define a smoothed variant of the V-line transform as follows
\begin{equation} \label{smoothvline}
\begin{split}
\widetilde{ \mathcal{V}_{a,b}} f(\vx,\phi) &= a\widetilde{Lf}(\vx,\Phi) +b \widetilde{Lf}(\vx,\Phi')\\
&= \int_{\mathbb{R}^2}\varphi(\vx - \vx_0) \mathcal{V}_{a,b} f(\vx_0,\phi) \mathrm{d}\vx_0,
\end{split}
\end{equation}
where $\widetilde{Lf}(\vx,\phi) = \int_{\mathbb{R}^2} \varphi(\vx  -\vx_0) Lf(\vx_0,\phi) \mathrm{d}\vx_0$, and $\varphi > 0$ is a smoothing kernel which satisfies $\int_{\mathbb{R}^2}\varphi = 1$.
\end{definition}

\section{The scanning geometry and physical model} 
\label{model}
In this section, we present our scanning geometry and derive our physical model for the Compton scatter intensity. This leads us to a non-linear variant of the attenuated Radon transform.

We consider the scanning geometry in figure \ref{fig1}. The photons are emitted from $\vs$ with fixed energy $E$ ($\vs$ is monochromatic), and scatter from $\vx$ with energy $E_s$ towards $\vd$. $\vs$ and $\vd$ are constrained to lie on a lines tangent to the unit ball, $B = \{|\vx| = 1\}$, and the scanned object, $f \in L^2_c(B)$, is rotated with angle $\theta$ about the origin. The scattered energy is
\begin{equation}
\label{E_s}
E_s=\frac{E}{1+(E/E_0)(1-\cos\omega)},
\end{equation}
where $E_0\approx 511\text{keV}$ denotes the electron rest energy and $\omega = \pi - 2\psi$ is the scattering angle. By equation \eqref{E_s}, $\omega$ is determined by $E, E_s$. We consider $\omega \leq 90^\circ$ based on the geometry of figure \ref{fig1}, i.e., we consider forward scattered photons. The detectors are collimated to accept photons in direction perpendicular to the detector array. Thus, the photons are constrained to scatter on lines $L(s,\theta)$ as pictured in figure \ref{fig1}.

This will be important later, as we use known results on the classical (hyperplane) Radon transform in 2-D to prove some of our main theorems.
\begin{figure}[!h]
\centering
\begin{tikzpicture}[scale=0.8]
%\draw[fill=green,rounded corners=1mm] (0,0) \irregularcircle{2cm}{0.5mm};
\draw[rotate=30, fill=green!50,rounded corners=1mm] (0,0) ellipse [x radius=2cm, y radius=1cm];
\node at (-1.5,-0.5) {$f$};
\draw [->,line width=1pt] (-6,0)--(6,0)node[right] {$x_1$};
\draw [->,line width=1pt] (0,-6)--(0,6)node[right] {$x_2$};
\draw [thick] (0,0) circle (4);
\draw [<->] (0,0)--(2.83,-2.83);
\node at (1.415,-1.8) {$1$};
\draw [red] (4,4)--(-4,4)node[left]{\textcolor{black}{source array}};
\draw [blue] (4,4)--(4,-4)node[below]{\textcolor{black}{detector array}};
\draw (4,1)--(-5,1)node[above]{$L(s,\theta)$};
\draw (1,4)node[above]{$\vs$}--(1,1)node[below]{$\vx$};
\draw (1.2,1)--(4,1)node[right]{$\vd$};
\node at (0.77,2.5) {$l_1$};
\node at (2.5,0.75) {$l_2$};
\coordinate (D) at (1,4);
\coordinate (w) at (1,1);
\coordinate (a) at (4,1);
\draw pic[draw=orange, <->,"$2\psi$", angle eccentricity=1.5] {angle = a--w--D};
\end{tikzpicture}
\caption{Linear CST scanner design. Sources ($\vs$) along the red half line $\{(1 - t,1) : t\geq 0\}$ emit photons energy $E$ and illuminate a density, $f$, supported on $B = \{|\vx|<1\}$. The incoming photons attenuate along $l_1$, scatter at $\vx$ with energy $E_s$, and then attenuate along $l_2$ towards the detector, $\vd$ (on the blue detector array). The detectors are collimated in the $x_1$ direction (perpendicular to the detector array) to receive photons scattered on the line $L(s,\theta)$. $f$ is rotated by angle $\theta$ about the origin. In the picture, $\theta = \pi/2$.}
\label{fig1}
\end{figure}
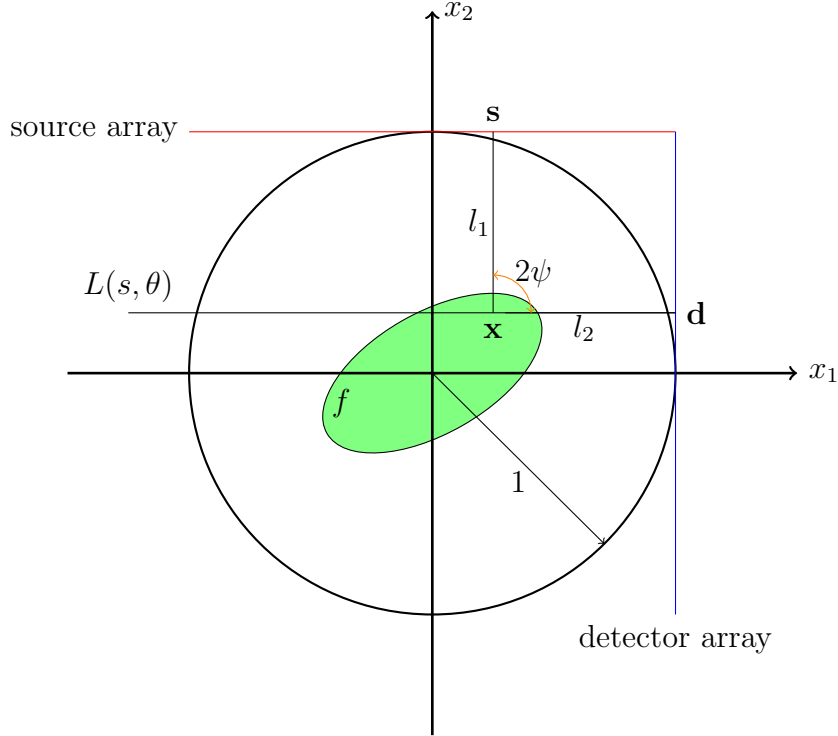

The first-order Compton signal can be modeled
\begin{equation}
\label{model_0}
\mathcal{R}f(s,\theta) = \lambda(E,\omega)\int_{\vx\in L(s,\theta)}w f(\vx)e^{-\paren{ \widetilde{L \mu_E}(\vx, \theta + 2\psi - \pi/2) + \widetilde{L\mu_{E_s}}(\vx, \theta - \pi/2)}}\mathrm{d}l,
\end{equation}
where $f$ denotes the electron density, $\mu_E$ is the attenuation coefficient at energy $E$, and
\begin{equation}
L(s,\theta) = \{\vx\in\mathbb{R}^2 : \vx \cdot \Theta = s\},
\end{equation}
with $\Theta = (\cos\theta,\sin\theta)$. The arc element on $L(s,\theta)$ is denoted $\mathrm{d}l$. The smoothing kernel, $\varphi$, which is implicit in the definition of $\widetilde{L\mu_E}$ (see Definition \ref{vline_def}), models the finite size of $\vs$ and $\vd$. {Similar models have been used previously in the literature \cite{ziou2021scale}.} For example, $\varphi$ could be a characteristic on a unit ball, assuming a spherical pixel. When $\varphi$ is a delta function, then $\vs$ and $\vd$ are point sources and detectors. Later, in section \ref{main_thm_section}, we will establish conditions on $\varphi$ to prove injectivity results for $\mathcal{R}$. 

The remaining terms in \eqref{model} are defined as follows. The smooth weight, $w$, accounts for solid angle effects \cite{me2} and satisfies $w > 0$ on $B$.
The $\lambda$ term is defined
\begin{equation}
\lambda = \lambda(E,\omega) = I_0(E)\frac{\mathrm{d}\sigma}{\mathrm{d}\Omega}(E,\omega),
\end{equation}
where $I_0$ is the photon intensity from $\vs$, and $\frac{\mathrm{d}\sigma}{\mathrm{d}\Omega}$ is the differential cross section, as given by the Klein-Nishina formula \cite{klein1929streuung}. For the purposes of our theorems, we fix $\omega \leq 90^{\circ}$, and thus, since $E$ is also fixed, $\lambda$ is essentially constant.
The line segments $l_1 = l_1(\vs,\vx)$ and $l_2 = l_2(\vx,\vd)$ connect $\vs$ to $\vx$, and $\vx$ to $\vd$, respectfully, as labeled in figure \ref{fig1}. %The line elements on $l_1$ and $l_2$ are denoted $\mathrm{d}l_1$ and $\mathrm{d}l_2$.

The attenuation coefficient and electron density satisfy \cite{WebberQuinto2020II}
\begin{equation}
\mu_E(Z) = \sigma_E(Z)\cdot f,
\end{equation}
where $\sigma_E$ is the total electron cross-section, and $Z$ is the material atomic number. When $E$ becomes large enough ($E\geq 100$keV) and $Z$ is small enough ($Z\leq 20$), then $\sigma_E$ loses it's dependence on $Z$, i.e., $f$ and $\mu_E$ are proportional with constant of proportionality $\sigma_E$ \cite{WebberQuinto2020II}.

Let us assume the source energy is sufficiently large so that $E, E_s \geq 100$keV. For example, if $E\geq 1$MeV, then $E_s\geq 338$keV since $\omega \leq \pi/2$. This assumption is reasonable as many monochromatic sources (e.g., gamma ray) have energies at MeV magnitude. Let us further assume that the material effective atomic number satisfies $Z\leq 20$ (e.g., any carbon, hydrogen, nitrogen, or  oxygen-based molecule). Then, \eqref{model} becomes
\begin{equation}
\label{model_1}
\mathcal{R}f(s,\theta) = \lambda\int_{L(s,\theta)}w fe^{-\widetilde{ \mathcal{V}_{a,b} }f(\vx, \theta)}\mathrm{d}l,
\end{equation}
where $a = \sigma_E >0 $ and $b = \sigma_{E_s}>0$. This is a non-linear operator acting on $f$ due to the exponent and as the integral weights depend on $f$. This can be viewed as a non-linear variant of the classical Radon transform. We shall now analyze the properties of the singularities of $\mathcal{R}f$ in the following sections.

\section{Main microlocal theorem}
\label{main_thm_section}
In this section, we present our central microlocal theory which analyzes the Sobolev regularity of the singularities of $\mathcal{R}f$. Let us first define the classical weighted (linear) Radon transform
\begin{equation}
R_wf(s,\theta) = \int_{L(s,\theta)} wf \mathrm{d}l,
\end{equation}
where $f\in L^2_c(B)$ and $w$ is a smooth weight with $w > 0$ on $B$. Using a combination of classical microlocal analysis and Sobolev space theory, we will show that, for $\varphi$ sufficiently smooth (i.e., in a high enough order Sobolev space) and $f = u\chi_\Omega$,
\begin{equation}
\label{main_res}
\text{ssupp}\paren{\mathcal{R}f} \backslash \text{ssupp}^{1}\paren{\mathcal{R}f} = \text{ssupp}\paren{R_wf},
\end{equation}
where $\chi_\Omega$ is the characteristic function on $\Omega \subset \{|\vx| <R\}$, a simply connected domain, and $u>0$ is a smooth function. We use this result to show, under reasonable assumptions on the boundary of $\Omega$, that $\WF(f)$ is uniquely determined by $\mathcal{R}f$.
%\jc{I think we may need the boundary of $\Omega$ to have non-zero curvature to get the full wavefront set.}
%\seanc{Is the last sentence required?}
%This means that the least regular singularities of $\mathcal{C}f$ correspond to those of $R_wf$. It is well known that $\WF\paren{R_wf}$ determines uniquely $\WF(f)$ (\tred{reference}). Thus, \eqref{main_res} is important as it shows we can recover $\WF(f)$ uniquely from $\mathcal{C}f$.

To prove \eqref{main_res}, we will first reformulate the physical model \eqref{model_1} using FIO and distributional products.

\subsection{Reformulation of physical model}
Let us define the linear operator $\mathcal{A}_w : L^2_c(\mathbb{R}^3) \to L^2(\mathbb{R} \times [-\pi,\pi])$,
\begin{equation}
\mathcal{A}_wh(s,\theta) = \int_{\mathbb{R}^3}\delta( \theta - x_3)\delta(\vx'\cdot \Theta - s) w h(\vx) \mathrm{d}\vx,
\end{equation}
where $w>0$ is a smooth weight and $\vx' = (x_1,x_2)$. Then, $\mathcal{A}$ defines the integrals of $h$ over lines embedded in planes $\{x_3 = \theta\}$ and $\mathcal{A}_wh = R_w(h(\cdot,\theta))(s,\theta)$.

Let $g = e^{-\widetilde{\mathcal{V}_{a,b}}f}$. Then, \eqref{model_1} becomes
\begin{equation}
\label{R_to_R}
\mathcal{R}f(s,\theta) = \mathcal{A}_w\paren{ f g }(s,\theta) = R_w(h(\cdot,\theta))(s,\theta),
\end{equation}
where $h = fg$. Thus, the Compton signal can be modeled as a linear ray transform applied to the product $fg$. For each fixed $\theta$, this defines the standard Radon projection of $h$ in direction $\Theta$ on the 2-D slice $h(\cdot,\theta)$.

We now discuss the microlocal and smoothing properties of the V-line transform and go on to analyze the smoothness of $h$.

\subsection{The wavefront set of the V-line transform}
Let us consider properties of the divergent beam transform
\[
\begin{split}
Lf(\vx,\Phi) & = \int_0^\infty f(\vx+t \Phi) \d t\\
& = \int_{-\infty}^\infty f(\vx + (t- \Phi^T \vx) \Phi) H_{\Phi^T \vx}(t) \ \d t
\end{split}
\]
where $H_{\Phi^T \vx}(t)$ is the Heaviside function with step at $\Phi^T \vx$. If we define
\[
h(\vx,\Phi,t) =  f(\vx + (t - \Phi^T \vx) \Phi) H_{\Phi^T \vx}(t)
\]
then the divergent beam transform is equivalent to
\[
\mathcal{L}h(\vx,\Phi) = \int_{-\infty}^\infty h(\vx,\Phi,t) \ \mathrm{d} t.
\]
% We can analyse $g$ as a product of conormal distributions. Indeed, the Heaviside function can be written as
% \[
% H(s) = \frac{i}{2\pi} \int_{-\infty}^\infty e^{i s \sigma} \frac{\varphi(\sigma)}{\sigma} \ \mathrm{d} \sigma + h(s)
% \]
% where $\varphi \in C^\infty(\mathbb{R})$ is zero on a neighborhood of the origin and equal to $1$ away form the origin, and $h \in C^\infty(\mathbb{R})$. Therefore
% \[
% H_{\Phi^T \vx}(t) = \frac{i}{2\pi} \int_{-\infty}^\infty e^{i (t - \Phi^T \vx) \sigma} \frac{\varphi(\sigma)}{\sigma} \ \mathrm{d} \sigma + h(t - \Phi^T \vx)
% \]
% is in $I^{-3/2}(\mathbb{R}^2\times \mathbb{S}^1 \times \mathbb{R},N^*(Y)\setminus \{0\})$ (\tred{define these somewhere in section 2}) where
% \[
% Y = \{(\vx,\Phi,t) \in \mathbb{R}^2 \times \mathbb{S}^1\times \mathbb{R} \ : \ t - \Phi^T \vx = 0\}.
% \]
% If $f$ is a conormal distribution, the other factor, $f(\vx+(t-\Phi^T \vx)\Phi))$ will also be conormal with the same order.

%\seanc{At some point, I thought to use some results on products of conormal distributions here, but I'm not sure it's necessary now. I am commenting out this (unfinished) material.}

% The simple transform $\mathcal{L}$ integrating along one direction is in fact an FIO with canonical relation
% \[
% \Lambda_{\mathcal{L}} = \mathrm{span}\left \{ \d y + \d \widetilde{y}, \d \Phi + \d \widetilde{\Phi} \right \}
% \]
% and so only the singularities in $g$ normal to the $t$ axis will appear. 

Now consider the singularities of $h$. First, by \cite[Theorem 8.2.4]{hormanderI}
\[
\begin{split}
& \WF(f(\vx+(t-\Phi^T\vx) \Phi)) \subset \Big \{ \big (\vx,\Phi,t; (v - \Phi (\Phi^Tv)) \cdot \d \vx + ((t-\Phi^T\vx) (\Phi_\perp^T v) \\
& \hskip2cm -(\Phi_\perp^T \vx) (\Phi^T v) \d \phi + \Phi^T v \d t\big )\  : \ (\vx + (t - \Phi^T \vx) \Phi; v \cdot \d \vx) \in \WF(f) \Big \}
\end{split}
\]
where $\Phi = (\cos(\phi+2\psi-\pi/2),\sin(\phi+2\psi-\pi/2))^T$ and $\Phi_\perp = (-\sin(\phi+2\psi-\pi/2),\cos(\phi+2\psi-\pi/2))^T$. By the same theorem,
\[
\WF(H_{\Phi^T \vx}(t)) = \Big \{ \big (\vx,\Phi,\Phi^T \vx; a(\Phi \cdot \d \vx + (\Phi_\perp^T \vx)\d \phi - dt)\big ) \ ; \ a \in \mathbb{R}\setminus \{0\} \Big \}
\]
Then, by \cite[Theorem 8.2.10]{hormanderI}
\[
\begin{split}
& \WF(h) \subset \Big \{ \big (\vx, \Phi, \Phi^T \vx ; (a \Phi + (v - \Phi (\Phi^T v))) \cdot \d \vx + (a (\Phi_\perp^T \vx) - (\Phi_\perp^T \vx) (\Phi^T v)) \d \phi + (\Phi^T v - a) \d t   \big ) \\
& \hskip1cm : \ a \in \mathbb{R}\setminus \{0\}, \ (\vx; v \cdot \d \vx) \in \WF(f) \Big \} \bigcup \WF(f(\vx+(t-\Phi^T\vx) \Phi)) \bigcup \WF(H_{\Phi^T \vx}(t)).
\end{split}
\]
The operator $\mathcal{L}$ is an FIO with canonical relation given by
\[
\mathcal{C}_\mathcal{L} = \Big \{ \big ((x,\Phi; v\cdot \dd \vx + \alpha \dd \phi); ( x,\Phi, t ; v\cdot \dd \vx + \alpha \dd \phi) \big ) \Big \}
\]
Applying this canonical relation to the wavefront set gives
\[
\begin{split}
& \WF(\mathcal{L}h) \subset \mathcal{C}_{\mathcal{L}} \WF(h) = \Big \{(\vx,\Phi; v \cdot \d \vx) \ : (\vx; v \cdot \d \vx ) \in \WF(f) \Big \}
\\ &\hskip2cm \bigcup \Big \{ (\vx,
\Phi; v \cdot \d \vx + (t-\Phi^T\vx) (\Phi_\perp^T v) \d \phi) \ : \ (\vx+ (t - \Phi^T \vx) \Phi; v \cdot \d \vx) \in \WF(f), \ \Phi^T v = 0 \Big \}.
\end{split}
\]
We can actually improve this since $h(\vx,\Phi,t)$ is identically zero when $t<\Phi^T \vx$ and so the above set can be restricted to $t \geq \Phi^T \vx$. This gives
\[\label{div_wf}
\begin{split}
& \WF(\mathcal{L}h) \subset W :=\Big \{(\vx,\Phi; v \cdot \d \vx) \ : (\vx; v \cdot \d \vx ) \in \WF(f) \Big \}
\\ &\hskip.25cm \bigcup \Big \{ (\vx,
\Phi; v \cdot \dd \vx+ \alpha (\Phi_\perp^T v) \d \phi) \ : \  (\vx+ \alpha \Phi; v \cdot \d \vx) \in \WF(f), \ \Phi^T v = 0, \ \alpha \geq 0 \Big \}.
\end{split}
\]
Let us also define $W'$ to be $W$ but with $\Phi$ replaced by $\Phi' = (\cos(\phi-\pi/2),\sin(\pi-\pi/2))^T$ (recall \eqref{vline} and following text). Using this notation, we have now proven the following theorem.
\begin{theorem}\label{v_wf}
The wavefront set of the V-line transform $\mathcal{V}_{a,b}f$, for $a,b\neq 0$ satisfies
\begin{equation}\label{VWF}
\WF\paren{ \mathcal{V}_{a,b}f } \subset W \cup W'.
\end{equation}
\end{theorem}

Note that $\widetilde{V_{a,b}}$ (see \eqref{smoothvline}) is defined to be $\mathcal{V}_{a,b}$ convolved with a kernel in $\vx$, and so does not spread the wave front set. Thus \eqref{VWF} also holds with $\mathcal{V}_{a,b}$ replaced by $\widetilde{\mathcal{V}_{ab}}$. Now let $\tilde{f}(\vx,\phi) = f(\vx)$ be the natural extension of $f$ to $\mathbb{R}^3$. By \eqref{div_wf} and Theorem \ref{v_wf}, it is possible for $\widetilde{V_{a,b}}f$ and hence $g = e^{-\widetilde{\mathcal{V}_{a,b}}f}$ to have singularities which occur in the same direction as $\tilde{f}$. In \cite{hormanderI}, wavefront sets of distributional products $fg$ are characterized in the case when $f$ and $g$ don't have singularities in the same direction. Thus, the classical theory of \cite{hormanderI} does not apply here and we require expansions of the theory to address Sobolev regularity of $fg$. We will do this later in section \ref{product_section}, using the generalized Sobolev spaces $H^{\alpha,r}$ defined in Definition \ref{gen_sobo}.

\subsection{Smoothing properties of $\mathcal{V}_{a,b}$}

We now analyze the smoothing properties of $\mathcal{V}_{a,b}$ on Sobolev scale.

Let 
\begin{equation}
\mathcal{V}_L f(\vx,\phi) = a\int_0^\nu f(\vx + t\Phi) \mathrm{d}t + b\int_0^\nu f(\vx + t\Phi'),
\end{equation}
for some $\nu > 0$, where $\Phi, \Phi'$ are as in Definition \ref{vline_def}. When $\nu\geq 2$, since $\text{supp}(f) \subset \{|\vx| < 1\} = B$, for $\vx\in B$, $\mathcal{V}_Lf(\vx,\phi) = \mathcal{V}_{a,b}f(\vx,\phi)$ for any fixed $\phi$. For $\nu$ large enough (depending on the size of the smoothing kernel $\varphi$), we have $\paren{\varphi \ast \mathcal{V}_Lf(\cdot,\phi)}(\vx) =  \widetilde{\mathcal{V}_L}f(\vx,\phi) =  \widetilde{\mathcal{V}_{a,b}}f(\vx,\phi)$ for $\vx \in B$. Therefore $fg = fe^{- \widetilde{\mathcal{V}_L}f} = fe^{- \widetilde{\mathcal{V}_{a,b}}f}$ and, for the purposes of this work where we are interested in smoothness of the product $fg$, we can equivalently consider the smoothness of $g = e^{- \widetilde{\mathcal{V}_L}f}$. 

We first discuss the smoothing properties of $V_Lf$. We have
\begin{equation} \label{Vlhat}
\widehat{\mathcal{V}_L f}(\xi,\phi) = \hat{f}(\xi)\paren{ a u(\xi\cdot\Phi) + b u(\xi\cdot \Phi') }= \hat{f}(\xi) w(\xi),
\end{equation}
where
\begin{equation}
u(y) = \int_0^\nu e^{i t y} \mathrm{d}t = -i\frac{e^{i \nu y} - 1}{y}.
\end{equation}
$u$ is bounded by $\nu$ and is $\mathcal{O}(|y|)$ as $|y| \rightarrow \infty$. For example, when $\nu = 1$, $u(y) = \sinc(y) + i\sin(\frac{y}{2}) \sinc(\frac{y}{2})$. We now have the following theorem.

%\seanc{I think there were some sign errors in the paragraph above which I've changed (e.g. it is now $e^{ity}$ rather than $e^{-ity}$). Also, I'm not sure we've defined what exactly we mean by 'decay at rate ...' so I changed the statement here and plan to below. I think it unifies things a bit if we say instead $f\in H^{\beta,\infty}(\mathbb{R}^2)$ which I believe is what we want.}

\begin{theorem}
\label{V_smooth}
Let $\phi \in [-\pi,\pi]$ be fixed and suppose $f\in H^{\beta,\infty}(\mathbb{R}^2)$. Then, $\mathcal{V}_L f (\cdot,\phi) \in H^{\alpha,2}(\mathbb{R}^2)$ for any $\alpha < \beta - 1/2$.
\end{theorem}
\begin{proof}
%\seanc{I think $\Phi^\perp$ should be $\Phi'$ in some places below (but not all). Iam making the changes, and also changing $\Phi^\perp$ to $\Phi_\perp$ to be consistent with the previous section.}
We have
\begin{equation}
\label{equ_V}
\|\widehat{\mathcal{V}_L f}(\cdot,\phi)\|^2_{H^{\alpha,2}(\mathbb{R}^2)} \leq a\int_{\mathbb{R}^2}|\hat{f}(\xi)|^2|u(\xi\cdot\Phi)|^2(1+|\xi|^2)^\alpha\mathrm{d}\xi+ b\int_{\mathbb{R}^2}|\hat{f}(\xi)|^2|u(\xi\cdot\Phi')|^2(1+|\xi|^2)^\alpha\mathrm{d}\xi.
\end{equation}
Let's focus on the first term. We have
\begin{equation}
\begin{split}
\int_{\mathbb{R}^2}|\hat{f}(\xi)|^2|u(\xi\cdot\Phi)|^2(1+|\xi|^2)^\alpha &= \int_{\mathbb{R}^2}|\hat{f}(\xi)|^2\frac{|e^{i\nu \Phi \cdot \xi} - 1|^2}{|\Phi \cdot \xi|^2} (1+|\xi|^2)^\alpha \mathrm{d}\xi\\
& = \int_{\mathbb{R}^2}|\hat{f}(Q\xi)|^2\frac{|e^{i\nu \xi_1} - 1|^2}{\xi_1^2} (1+|\xi|^2)^\alpha\mathrm{d}\xi,
\end{split}
\end{equation}
where $Q = [\Phi,\Phi_\perp]$ and we have made the substitution $\xi \to Q \xi$. 

Making the substitution $\xi_2 = (1 + \xi_1^2)^{\frac{1}{2}}t$ yields
\begin{equation}
\begin{split}
\int_{\mathbb{R}^2}&|\hat{f}(Q(\xi_1,\sqrt{1+\xi_1^2}t)^T)|^2\frac{|e^{i\nu \xi_1} - 1|^2}{\xi_1^2}(1+\xi_1^2)^{\alpha+1/2}(1+t^2)^\alpha\mathrm{d}\xi_1 \mathrm{d}t%\\
%&\leq C + \int_{\mathbb{R}^2\backslash[-r,r]^2}|\hat{f}(Q(\xi_1,\sqrt{1+\xi_1^2}t)^T)|^2\frac{|e^{i\nu \xi_1} - 1|^2}{\xi_1^2}(1+\xi_1^2)^{\alpha+1/2}(1+t^2)^\alpha\mathrm{d}\xi_1 \mathrm{d}t,
\end{split}
\end{equation}
for any $r\geq 0$. Since $f\in H^{\beta,\infty}(\mathbb{R}^2)$ we have
$$|\hat{f}(\xi)|^2\leq c\frac{1}{(1+|\xi|^2)^\beta}$$
for almost every $\xi$. Then, the integral is bounded by
\begin{equation}
\begin{split}
c\int_{\mathbb{R}^2}\frac{1}{(1+\xi_1^2)^\beta(1+t^2)^\beta}\frac{|e^{i\nu \xi_1} - 1|^2}{\xi_1^2}(1+\xi_1^2)^{\alpha+1/2}(1+t^2)^\alpha\mathrm{d}\xi_1 \mathrm{d}t
\end{split}
\end{equation}
for some $c>0$. This converges if and only if
$$\int_0^\infty(1+t^2)^{\alpha - \beta}\mathrm{d}t$$
converges, which happens when $\alpha < \beta - 1/2$. The second term in \eqref{equ_V} can be bounded similarly.
\end{proof}
For example, if $f \in H^{1,\infty}(\mathbb{R}^2)$ (e.g., if $f = \chi_\Omega$ and $\Omega$ has a flat boundary), then $\mathcal{V}_L f(\cdot,\phi) \in H^{\alpha,2}(\mathbb{R}^2)$ for any $\alpha <1/2$. Such $f$ are of the same Sobolev order (see Theorem \ref{char_sobo}). If instead $f \in H^{3/2,\infty}(\mathbb{R}^2)$ (which is true if $\partial \Omega$ has non-zero curvature \cite{f1,f2}), then $\mathcal{V}_L f(\cdot,\phi) \in H^{\alpha,2}$ for any $\alpha <1$. Thus, in this case $\mathcal{V}_L$ smoothes $f$ by order $1/2$. However, unless we make assumptions about $\partial \Omega$, $\mathcal{V}_L$ may not smooth $f$ at all, which makes sense as when one of the rays of the V-line is tangent to a flat piece of $\partial \Omega$, a jump-type singularity can occur in $\mathcal{V}_Lf$. To formalize this further as a counterexample, let $f = \chi_{[-0.5,0.5]^2}$, let $\nu = 2$, let $\phi = \pi/2$ and $\psi = \pi/4$ (i.e., the opening angle of the V-line is $\pi/2$). Let $u \in C_0^\infty(B_\delta((0.5,-1)))$ be a smooth cutoff centered on $(0.5,-1)$. Then, for $\delta$ small enough, $u \mathcal{V}_L f = u \chi_{\{x_1 < 0.5\}}$, and thus $\mathcal{V}_L f$ is not in $H^{1/2}$ near $(0.5,-1)$ as the characteristic $\chi_{\{x_1 < 0.5\}}$ is not in $H^{1/2}$ near $(0.5,-1)$ (see Theorem \ref{char_sobo}). Therefore, in this example, $\mathcal{V}_L f$ is not in $H^{1/2}(\mathbb{R}^2)$ and is of the same Sobolev order as $f$. For further analysis on the smoothing properties of the V-line transform which may be of interest to the reader, see appendix \ref{add_vline}.

\begin{corollary}
Let $\phi \in [-\pi,\pi]$ be fixed, let $f \in H^{\beta,\infty}(\mathbb{R}^2)$ and let $\varphi \in H^{\lambda,\infty}(\mathbb{R}^2)$ (recall $\varphi$ is the kernel used in the definition of $\widetilde{\mathcal{V}_{a,b}}$, see \eqref{smoothvline}). Then, $\widetilde{\mathcal{V}_L} f (\cdot,\phi) \in H^{\alpha,2}(\mathbb{R}^2)$ for any $\alpha < \beta + \lambda - 1/2$.
\end{corollary}
\begin{proof}
This is a simple consequence of Theorem \ref{V_smooth}.
\end{proof}

We now have a variant of this theorem which relates standard Sobolev regularity of $f$ to that of $\widetilde{\mathcal{V}_L f }$.
\begin{theorem}
\label{thm_V_smth}
Let $\phi \in [-\pi,\pi]$ be fixed, $\varphi \in H^{\lambda,\infty}(\mathbb{R}^2)$, and $f \in H^{\alpha,2}(\mathbb{R}^2)$. Then, $\widetilde{\mathcal{V}_L f }(\cdot,\phi) \in H^{\alpha + \lambda,2}(\mathbb{R}^2)$.
\end{theorem}
\begin{proof}
We have (recall \eqref{Vlhat})
\begin{equation}
\widehat{\widetilde{\mathcal{V}_L}f}(\xi,\phi) = \hat{\varphi}(\xi)\hat{f}(\xi) w(\xi)
\end{equation}
and $|\widehat{\widetilde{V_L}f}(\xi,\phi)| \leq c |\hat{\varphi}(\xi)| |\hat{f}(\xi)|$, since $w$ is bounded.

Let $\widetilde{V_L}f = \widetilde{V_L}f(\cdot,\phi)$. We have
\begin{equation}
\begin{split}
\|\widetilde{\mathcal{V}_L}f\|^2_{H^{\alpha +\lambda,2}(\mathbb{R}^2)} &\leq c \int_{\mathbb{R}^2}|\hat{\varphi}(\xi)|^2|\hat{f}(\xi)|^2(1+|\xi|^2)^{\alpha+\lambda}\mathrm{d}\xi\\
& \leq c' \int_{\mathbb{R}^2}|\hat{f}(\xi)|^2(1+|\xi|^2)^{\alpha}\mathrm{d}\xi\\
& \leq c' \|f\|^2_{H^{\alpha,2}(\mathbb{R}^2)}
%&= C + c\int_{\mathbb{R}^2\backslash B_r(0)}\hat{\varphi}(\xi)^2|\hat{f}(\xi)|^2(1+|\xi|^2)^\alpha\mathrm{d}\xi.
\end{split}
\end{equation}
% Thus, for $r$ large enough,
% \begin{equation}
% \begin{split}
% \|\widetilde{\mathcal{V}_L}f\|_{H^{\alpha,2}(\mathbb{R}^2)}& \leq C + c\int_{\mathbb{R}^2\backslash B_r(0)}|\hat{f}(\xi)|^2\frac{(1+|\xi|^2)^\alpha}{|\xi|^{2\lambda}}\mathrm{d}\xi\\
% &\leq c' \|f\|_{H^{\alpha - \lambda,2}(\mathbb{R}^2)}
% \end{split}
% \end{equation}
This completes the proof.
\end{proof}

%\seanc{Shortened this proof since using $\varphi \in H^{\lambda,\infty}$}

\subsection{Regularity of distributional products}
\label{product_section}
In this section, we calculate the regularity of $fg$ using the modified Sobolev spaces introduced in Definition \ref{gen_sobo}. We now prove a number of results regarding the $H^{\alpha,p}$ spaces leading up to our main theorem regarding distributional products.

\begin{lemma} \label{lem:Hsa_cont}
    Suppose $1 \leq p$, $\widetilde{p} \leq \infty$ and
\begin{equation}\label{eq:bcond}
    n\frac{\widetilde{p} - 1}{\widetilde{p}p}< \beta
\end{equation}
where we define $\infty/\infty = 1$ and $\infty/\infty^2 = 0$.
Then
\[H^{\alpha+\beta,\widetilde{p}p}(\mathbb{R}^n) \subseteq H^{\alpha,p}(\mathbb{R}^n).
\]
\end{lemma}
\begin{proof}
    Let us first suppose $1< p, \widetilde{p}<\infty$. If
    \begin{equation}\label{eq:qt}
    \widetilde{q}= \frac{\widetilde{p}}{\widetilde{p}-1} \Leftrightarrow \frac{1}{\widetilde{p}}+ \frac{1}{\widetilde{q}} = 1,
    \end{equation}
    then using H\"older's inequality
    \[
    \begin{split}
    \|\widehat{f}(1+|\xi|^2)^{\frac{\alpha}{2}}\|^p_p &  = \| \widehat{f} (1+|\xi|^2)^{\frac{\alpha+\beta}{2}} (1+|\xi|^2)^{-\frac{\beta}{2}}\|^p_p\\
    & = \| |\widehat{f}|^p (1+|\xi|^2)^{\frac{p(\alpha+\beta)}{2}} (1+|\xi|^2)^{-\frac{p\beta}{2}}\|_1\\
    & \leq \||\widehat{f}|^p (1+|\xi|^2)^{\frac{p(\alpha+\beta)}{2}}\|_{\widetilde{p}} \| (1+|\xi|^2)^{-\frac{p\beta}{2}}\|_{\widetilde{q}}\\
    & \leq \| |\widehat{f}|^{\widetilde{p}p} (1+|\xi|^2)^{\frac{\widetilde{p}p(\alpha+\beta)}{2}}\|_{1}^{1/\widetilde{p}}\| (1+|\xi|^2)^{-\frac{p\beta}{2}}\|_{\widetilde{q}}\\
    & \leq \|f\|_{H^{\alpha+\beta,\widetilde{p}p}}^p \| (1+|\xi|^2)^{-\frac{p\beta}{2}}\|_{\widetilde{q}}.
    \end{split}
    \]
    If
    \begin{equation}\label{eq:ineq}
    n-1-\widetilde{q}p\beta < -1 \Leftrightarrow n < \widetilde{q}p\beta,
    \end{equation}
    then $\| (1+|\xi|^2)^{-\frac{p\beta}{2}}\|_{\widetilde{q}}$ is finite. Using \eqref{eq:qt} in \eqref{eq:bcond} implies \eqref{eq:ineq} and so completes the proof for this case. 
    
    The special cases when $p$ or $\widetilde{p}$ equal to $1$ or $\infty$ can be proven in a similar way.
\end{proof}

\begin{theorem}\label{char_sobo}
    Let $\Omega \subset \mathbb{R}^n$ be a bounded set with smooth boundary and let $u\in C^\infty(\mathbb{R}^n)$, with $u > 0$. Then $u\chi_\Omega \in H^{\alpha,2}(\mathbb{R}^n)$ if and only if $\alpha < 1/2$ and in $H^{\alpha,1}(\mathbb{R}^n)$ if $\alpha < (1-n)/2$.
\end{theorem}
\begin{proof}
The first part of the theorem, i.e., $u\chi_\Omega \in H^{\alpha,2}(\mathbb{R}^n)$ for $\alpha < 1/2$, is a well known result regarding characteristic functions \cite[page 92]{natterer}. The second part follows from Lemma \ref{lem:Hsa_cont} setting $p = 1$, $\tilde{p} = 2$.
\end{proof}

We now have the theorem which addresses composition of $H^{\alpha,p}$ functions with the exponential function.

\begin{theorem}\label{thm_sep}
Let $f = u\chi_{\Omega}$, where $\Omega \subset B$ is a compact domain with smooth boundary, and $u \in C_c^\infty(\mathbb{R}^2)$ with $u > 0$ on $\Omega$. Let $\varphi\in H^{\lambda,\infty}(\mathbb{R}^2)$. Then, for any smooth cutoff function $\psi \in C_c^\infty(\mathbb{R}^2)$, $\psi g \in H^{\alpha + \lambda,2}(\mathbb{R}^2)$ and $\psi g \in H^{\alpha + \lambda - 1,1}(\mathbb{R}^2)$ for any $\alpha < 1/2$, where $g = e^{ -\widetilde{\mathcal{V}_L}f(\cdot, \phi) }$ and $\phi \in [-\pi, \pi]$ is fixed.
\end{theorem}
\begin{proof}
Let $\widetilde{\mathcal{V}_L}f = \widetilde{\mathcal{V}_L}f(\cdot, \phi)$. By Theorem \ref{char_sobo}, $f$ is in $H^{\alpha,2}(\mathbb{R}^2)$ for $\alpha <1/2$, and by Theorem \ref{thm_V_smth}, $\widetilde{\mathcal{V}_L}f \in H^{\alpha + \lambda,2}(\mathbb{R}^2)$. Now, applying Theorem \ref{sobo}, setting $\gamma(x) = e^{-x} : [0,\infty) \to [0,1]$, yields $\psi \paren{ \gamma \circ \paren{ \widetilde{\mathcal{V}_L}f } }= \psi g\in H^{\alpha + \lambda,2}(\mathbb{R}^2)$ for any $\alpha < 1/2$. The second part of the proof, namely that $\psi g \in H^{\alpha + \lambda - 1,1}(\mathbb{R}^2)$ for any $\alpha < 1/2$, follows from Lemma \ref{lem:Hsa_cont} setting $p = 1$, $\tilde{p} = 2$.
\end{proof}

\noindent We now have our main theorem regarding the regularity of distributional products.

\begin{theorem}\label{mainprod}
    Suppose that $f$ and $g$ are distributions on $\mathbb{R}^n$,  $\alpha_f + \alpha_g \geq 0$, $1 \leq p, q, r \leq \infty$ are such that
    \begin{equation}\label{eq:Young}
        \frac{1}{p} + \frac{1}{q} = \frac{1}{r} + 1,
    \end{equation}
    and $\vx_0 \in \mathbb{R}^n$. Then
    \begin{enumerate}
        \item
        \label{item:1.7.0}
        if $f \in H^{\alpha_f,p}(\mathbb{R}^n)$, $g \in H^{\alpha_g,q}(\mathbb{R}^n)$, then $fg$ is a well defined element of $H^{\min\{\alpha_f,\alpha_g\},r}(\mathbb{R}^n)$.
        % \item \label{item:1.7.1} if $f \in H^{\alpha,p}$ and $g \in H^{\alpha,q}$ near $\vx_0$,
        % then $fg \in H^{\alpha,r}$ near $\vx_0$;
        \item \label{item:1.7.2} if $\xi_0 \notin WF_{\vx_0}(f) \cup WF_{\vx_0}(g)$, and $f \in H^{\alpha_f,p}$ and $g \in H^{\alpha_g,q}$ near $\vx_0$, then $fg \in H^{\alpha_f + \alpha_g + \min\{\alpha_f,\alpha_g,0\},r}$ near $(\vx_0,\xi_0)$.
        \item \label{item:1.7.3} if $\xi_0 \notin WF_{\vx_0}(f)$, $\alpha_g \geq 0$, $f \in H^{\alpha_f,p}$ and $g \in H^{\alpha_g,q}$ near $\vx_0$, then $fg \in H^{\alpha_g+\min\{0,\alpha_f\},r}$ near $(\vx_0,\xi_0)$.
    \end{enumerate}
\end{theorem}

\begin{remark}
    Claims \eqref{item:1.7.2} and \eqref{item:1.7.3} can be refined by considering the microlocal regularity of $f$ and $g$ and using a similar method to the proof given here. Indeed, it is already known \cite[Theorem 8.2.10]{hormanderI} that if both $WF_{\vx_0}(f)$ and $WF_{\vx_0}(g)$ are nonempty and
    \begin{equation} \label{eq:0cond}
    0 \notin \{ \xi_f + \xi_g \ : \ \ \xi_f \in WF_0(f) \ \mbox{and} \ \xi_g \in WF_{\vx_0}(g) \},
    \end{equation}
    then
    \[
    WF_{\vx_0}(fg) \subset \overline{\{\xi_f + \xi_g \ : \ \xi_f \in WF_{\vx_0}(f) \ \mbox{and} \ \xi_g \in WF_{\vx_0}(g)\}} \setminus\{0\}.
    \]
%\jc{changed to subset above as that's how the theorem's stated in the book.}

\noindent which provides more detail in the location of $WF_{\vx_0}(fg)$ but requires \eqref{eq:0cond}. Note that \eqref{eq:0cond} is not a hypotheses of the theorem.
\end{remark}

\begin{proof}
    Throughout this proof we use the notation $C$ to denote a positive constant which may change from step-to-step. We assume \eqref{eq:Young} throughout.
    
    To prove claim \ref{item:1.7.0}, let us also assume, without loss of generality, that $\alpha_g \geq 0$ and $\alpha_g \geq \alpha_f$. Then, it is sufficient to show that
    \begin{equation}
        \label{eq:basic}
    \|fg\|_{H^{\alpha_f,r}(\mathbb{R}^n)} \leq C \|f\|_{H^{\alpha_f,p}(\mathbb{R}^n)} \|g\|_{H^{\alpha_g,q}(\mathbb{R}^n)}
    \end{equation}
    for all $f$, $g \in C_c^\infty(\mathbb{R}^n)$ and some $C>0$. For such $f$ and $g$, we have
    \begin{equation}\label{eq:est1}
    \begin{split}
    \|f g\|_{H^{\alpha_f,r}(\mathbb{R}^n)} & = (2 \pi)^n \|\hat{f}\ast \hat{g}(1+|\xi|^2)^{\alpha_f/2}\|_{L^r(\mathbb{R}^n)}\\
    & \leq (2\pi)^n  \Bigg (\int_{\mathbb{R}^n} \Bigg ( \int_{\mathbb{R}^n} (1+|\xi|^2)^{\frac{\alpha_f}{2}} (1+ |\xi-\eta|^2)^{-\alpha_f/2}\\ 
    & \hskip10em (1+ |\xi-\eta|^2)^{\alpha_f/2}|\hat{f}(\xi-\eta)| |\hat{g}(\eta)| \dd \eta \Bigg )^r \dd \xi \Bigg )^{1/r}.
    \end{split}
    \end{equation}
    Note that, by the triangle inequality,
    \[
    |\xi|^2\leq (|\xi-\eta| + |\eta|)^2 \leq 2 (|\xi-\eta|^2 + |\eta|^2)
    \]
    which implies
    \begin{equation}\label{eq:xiest1}
    1 + |\xi|^2 \leq 2(1+ |\xi-\eta|^2 + |\eta|^2) \leq 2 (1 + |\xi-\eta|^2) (1+ |\eta|^2).
    \end{equation}
    Therefore, if we assume $\alpha_f \geq 0$, we can establish \eqref{eq:basic} using \eqref{eq:est1}, $\alpha_g \geq \alpha_f$, and Young's inequality. On the other hand, using \eqref{eq:xiest1} with $\xi$ replaced by $\xi-\eta$ and $\eta$ replaced by $-\eta$ we have
    \begin{equation}\label{eq:xietaest}
    1+ |\xi-\eta|^2 \leq 2 (1+ |\xi|^2)(1+|\eta|^2).
    \end{equation}
    Applying this in \eqref{eq:est1} when $\alpha_f<0$, we obtain
    \[
    \begin{split}
    \|f g\|_{H^{\alpha_f,r}(\mathbb{R}^n)} & \leq C \Bigg (\int_{\mathbb{R}^n} \Bigg ( \int_{\mathbb{R}^n}(1+ |\xi-\eta|^2)^{\alpha_f/2}|\hat{f}(\xi-\eta)| (1+|\eta|^2)^{\alpha_g/2} |\hat{g}(\eta)| \dd \eta \Bigg )^r \dd \xi \Bigg )^{1/r}.
    \end{split}
    \]
    Note that we have used $\alpha_f + \alpha_g \geq 0$ which implies $\alpha_g \geq - \alpha_f \geq 0$. Application of Young's inequality then proves \eqref{eq:basic} which completes the proof of claim \eqref{item:1.7.0}.
    
    % Therefore, since $\alpha \geq 0$, \eqref{eq:tHarest} implies
    % \[
    % \begin{split}
    % \|u(D)\psi \widetilde{\psi} f g \|_{H^{\alpha,r}} & \leq C \left \| \big ((1+|\xi|^2)^{\alpha/2} \widehat{\psi f}\big )\ast \big ((1+|\xi|^2)^{\alpha/2} \widehat{\widetilde{\psi}g}\big )\right \|_r 
    % \end{split}
    % \]
    % Therefore, if we assume $\alpha_f \geq 0$, we can establish \eqref{eq:basic} using 
    
    % If $\alpha_f\geq 0$, then by Young's inequality this integral is bounded by $C\|f\|_{H^{\alpha_f,p}(\mathbb{R}^n)} \|g\|_{H^{\alpha_g,q}(\mathbb{R}^n)}$
    
    % Now we consider claim \eqref{item:1.7.2} and so assume the hypotheses of that claim and additionally that $\alpha_g \geq 0$ and $\alpha_g \geq \alpha_f$. 
    
    Suppose that $u \in C^\infty(\mathbb{R}^n \setminus\{0\})$
    is homogeneous of order $0$ and $\psi$, $\widetilde{\psi} \in C_c^\infty(\Omega)$
    are any functions such that $\psi(\vx_0)$, $\widetilde{\psi}(\vx_0)\neq 0$. Then, for any $\alpha$,
    \begin{equation} \label{eq:tHarest}
    \begin{split}
    \|u(D)\psi \widetilde{\psi} f g \|_{H^{\alpha,r}} & =(2\pi)^n  \| u(\widehat{\psi f}
    \ast \widehat{\widetilde{\psi} g}) ( 1+ |\xi|^2)^{\alpha/2} \|_{r}\\
    & = \left (\int_{\mathbb{R}^n} \left | \int_{\mathbb{R}^n} u(\xi) (1+|\xi|^2)^{\frac{\alpha}{2}} \widehat{\psi f}(\xi-\eta) \widehat{\widetilde{\psi} g}(\eta) \dd \eta \right |^r \dd \xi \right )^{1/r}.
    \end{split}
    \end{equation}
    We will use this formula as a basis for the proof of claims \eqref{item:1.7.2} and \eqref{item:1.7.3}.
    % Without making any other assumputions, by the triangle inequality
    % \[
    % |\xi|^2\leq (|\xi-\eta| + |\eta|)^2 \leq 2 (|\xi-\eta|^2 + |\eta|^2)
    % \]
    % which implies
    % \begin{equation}\label{eq:xiest1}
    % 1 + |\xi|^2 \leq 2(1+ |\xi-\eta|^2 + |\eta|^2) \leq 2 (1 + |\xi-\eta|^2) (1+ |\eta|^2).
    % \end{equation}
    % Therefore, since $\alpha \geq 0$, \eqref{eq:tHarest} implies
    % \[
    % \begin{split}
    % \|u(D)\psi \widetilde{\psi} f g \|_{H^{\alpha,r}} & \leq C \left \| \big ((1+|\xi|^2)^{\alpha/2} \widehat{\psi f}\big )\ast \big ((1+|\xi|^2)^{\alpha/2} \widehat{\widetilde{\psi}g}\big )\right \|_r 
    % \end{split}
    % \]
    % for a constant $C>0$. By Young's inequality for convolutions, for $p$ and $q$ as in \eqref{eq:Young}, this gives
    % \[
    % \|u(D)\psi \widetilde{\psi} f g \|_{H^{\alpha,r}} \leq C \| (1+|\xi|^2)^{\alpha/2} \widehat{\psi f} \|_{p} \| (1+|\xi|^2)^{\alpha/2} \widehat{\widetilde{\psi}g}\|_q = \|\psi f\|_{H^{\alpha,p}(\mathbb{R}^n)} \|\widetilde{\psi}g\|_{H^{\alpha,q}(\mathbb{R}^n)}.
    % \]
    % Since $\psi$, $\widetilde{\psi}$ and $u$ were not specified, this proves claim \eqref{item:1.7.1}. 
    Suppose $\xi_0 \in \mathbb{R}^n \setminus\{0\}$ and define $\theta_0 = \xi_0/|\xi_0|$. For $0<\epsilon<\pi/2$, consider the sets
    \[
    C_1 = \{\zeta \in \mathbb{R}^n \setminus \{0\} \ : \ \zeta \cdot \theta_0 \geq |\zeta| \cos(\epsilon) \}% \subset (WF_{\vx_0}(f) \cup WF_{\vx_0}(g))^c.
    \]
    Let
    \[
    C_0 = \{\zeta \in \mathbb{R}^n \setminus \{0\} \ : \ \zeta \cdot \theta_0 \geq |\zeta| \cos(\epsilon/2) \},
    \]
    \[
    C_2 = \{\zeta \in \mathbb{R}^n \setminus \{0\} \ : \ \zeta \cdot \theta_0 \leq |\zeta| \cos(\epsilon) \}
    \]
    and assume that $u$ in \eqref{eq:tHarest} is taken such that $\mathrm{supp}(u) \subset C_0$. Note that $C_1 \cup C_2 = \mathbb{R}^n \setminus \{0\}$ and the intersection has zero measure. We now prove a few inequalities that will be useful.
    \begin{lemma} \label{lem:conelem}
    If $\xi \in C_0$ and $\eta \in C_2$, then there is a constant $C>0$ such that
    \begin{equation} \label{eq:conelem1}
    |\xi| \leq C |\xi - \eta|.
    \end{equation}
    %If additionally $\xi-\eta \in C_2$, then there is a constant $C>0$ such that
    % \begin{equation}\label{eq:conelem2}
    % |\xi| \leq C \min\{|\eta|,|\xi-\eta|\}
    % \end{equation}
    \end{lemma}
    \begin{proof}
        Suppose that  $\xi \in C_0$ and $\eta \in C_2$. Then, by the formula defining $C_2$
        \[
        (\eta -\xi) \cdot \theta_0 + \xi \cdot \theta_0 \leq |\eta| \cos(\epsilon) \leq (|\eta - \xi| + |\xi|) \cos(\epsilon)
        \]
        which implies, since $\xi \in C_0$,
        \[
        |\xi| (\cos(\epsilon/2) - \cos(\epsilon)) \leq \xi \cdot \theta_0 - |\xi| \cos(\epsilon) \leq (\xi-\eta)\cdot \theta_0 + |\eta - \xi| \cos(\epsilon).
        \]
        Since $\cos(\epsilon/2)>\cos(\epsilon)>0$, this proves \eqref{eq:conelem1}.% Then \eqref{eq:conelem2} holds by applying \eqref{eq:conelem1} also with $\eta$ replaced with $\xi-\eta$.

        % Now suppose $\xi \in C_0$, $\xi - \eta \in C_2$ and $\eta \in C_2$. Then by adding the equations defining $\eta$ and $\xi- \eta \in C_2$ then using the equation for $C_1$, we have
        % \[
        % |\xi| \cos(\epsilon/2) \leq (|\eta| + |\xi - \eta|) \cos(\epsilon) \leq (2 |\eta| + |\xi|) \cos(\epsilon).
        % \]
        % This implies
        % \[
        % \xi(\cos(\epsilon/2) - \cos(\epsilon)) \leq 2 \cos(\epsilon) |\eta|.
        % \]
        % We can obtain the same estimate for $|\xi-\eta|$ and so this proves \eqref{eq:conelem2} since $\cos(\epsilon/2)>\cos(\epsilon)$.
    \end{proof}

    \noindent If we assume $\xi \in C_0$ and $\xi -\eta \in C_2$, note that we can replace $\eta$ by $\xi - \eta$ in this lemma and obtain
    \begin{equation}\label{eq:conelem2}
        |\xi| \leq C |\eta|.
    \end{equation}
    Now, let us split the integrand in \eqref{eq:tHarest} into a sum of four pieces using
    \begin{equation}\label{eq:partition}
    1 = \chi_{C_1}(\eta) \chi_{C_1}(\xi-\eta) + \chi_{C_2}(\eta) \chi_{C_1}(\xi-\eta) + \chi_{C_1}(\eta) \chi_{C_2}(\xi-\eta) + \chi_{C_2}(\eta)\chi_{C_2}(\xi-\eta)
    \end{equation}
    which for fixed $\xi$ holds outside a set of measure zero in $\eta$. Thus, we can bound \eqref{eq:tHarest} by a sum of four integrals corresponding to the four terms in this decomposition. We will bound each of these for the cases which occur in claims \eqref{item:1.7.2} and \eqref{item:1.7.3}.

    Let us now assume the hypotheses of claim \eqref{item:1.7.2}, and also assume as before, without loss of generality, that $\alpha_g \geq 0$ and $\alpha_g \geq \alpha_f$. Our objective is to prove \eqref{eq:tHarest} with $\alpha = \alpha_f + \alpha_g + \min\{\alpha_f,0\}$ is bounded. In this case, we can choose $\epsilon$ sufficiently small so that $C_1 \subset (\WF_{\vx_0}(f) \cup \WF_{\vx_0}(g))^c$ For the first term arising when we apply \eqref{eq:partition} to \eqref{eq:tHarest}, both $\eta$ and $\xi-\eta$ are in $(WF_{\vx_0}(f)\cup WF_{\vx_0}(g))^c$ on the support of the integrand, and so using \eqref{eq:xiest1} and Young's inequality for convolutions we prove this term is finite. 
    
    For the second term, we have $\xi \in C_0$, $\xi - \eta \in C_1$ and $\eta \in C_2$ and so we can apply \eqref{eq:conelem1} to estimate \eqref{eq:tHarest} by
    \begin{equation} \label{eq:2ndest}
    C \left (\int_{\mathbb{R}^n} \left ( \int_{\mathbb{R}^n}\chi_{C_2}(\eta) \chi_{C_1}(\xi-\eta) u(\xi) (1+|\xi-\eta|^2)^{\frac{\alpha_f + \alpha_g}{2}} |\widehat{\psi f}(\xi-\eta)| |\widehat{\widetilde{\psi} g}(\eta)| \dd \eta \right )^r \dd \xi \right )^{1/r}.
    \end{equation}
    Since $\xi - \eta \notin WF_{\vx_0}(f)$, this term is also finite by Young's inequality. Note we use $\alpha_g \geq 0$ here.
    
    For third term, $\xi \in C_0$, $\xi - \eta \in C_2$ and $\eta \in C_1$. Therefore, if $\alpha_f \geq 0$ we can estimate use \eqref{eq:conelem2} to estimate in a similar way as for the second term. If $\alpha_f <0$, we multiply and divide by $(1+|\xi-\eta|^2)^{\alpha_f/2}$ to bound by
    \[
    \begin{split}&
     C \Bigg (\int_{\mathbb{R}^n} \Bigg ( \int_{\mathbb{R}^n}\chi_{C_1}(\eta) \chi_{C_2}(\xi-\eta) u(\xi) (1+|\xi|^2)^{\frac{\alpha_f + \alpha_g}{2}} (1+|\xi-\eta|^2)^{\frac{-\alpha_f}{2}}\\
     &\hskip15em (1+|\xi-\eta|^2)^{\frac{\alpha_f}{2}}|\widehat{\psi f}(\xi-\eta)| |\widehat{\widetilde{\psi} g}(\eta)| \dd \eta \Bigg )^r \dd \xi \Bigg )^{1/r}.
     \end{split}
    \]
    We then apply \eqref{eq:xietaest}, \eqref{eq:conelem2}, the fact $\eta \notin \WF_{\vx_0}(g)$ on the support of the integrand, and Young's inequality to bound this term.

    For the fourth term $\xi \in C_0$, $\xi - \eta \in C_2$ and $\eta \in C_2$. If $\alpha_f \geq 0$ we can use both \eqref{eq:conelem1} and \eqref{eq:conelem2} to bound by
    \[
    \begin{split}
    C \left (\int_{\mathbb{R}^n} \left ( \int_{\mathbb{R}^n}\chi_{C_2}(\eta) \chi_{C_2}(\xi-\eta) u(\xi) (1+|\xi-\eta|^2)^{\frac{\alpha_f}{2}} |\widehat{\psi f}(\xi-\eta)|  (1+|\eta|^2)^{\frac{\alpha_g}{2}} |\widehat{\widetilde{\psi} g}(\eta)| \dd \eta \right )^r \dd \xi \right )^{1/r}
    \end{split}
    \]
    which is then bounded by Young's inequality. If $\alpha_f<0$, we first use \eqref{eq:conelem2} and multiply and divide by $(1+|\xi-\eta|^2)^{\alpha_f/2}$ to bound by
        \[
    \begin{split}
    & C \Bigg (\int_{\mathbb{R}^n} \Bigg ( \int_{\mathbb{R}^n}\chi_{C_2}(\eta) \chi_{C_2}(\xi-\eta) u(\xi) (1+|\xi|^2)^{\frac{\alpha_f}{2}} (1+|\xi-\eta|^2)^{-\frac{\alpha_f}{2}}\\
    & \hskip10em (1+|\xi-\eta|^2)^{\frac{\alpha_f}{2}}|\widehat{\psi f}(\xi-\eta)|  (1+|\eta|^2)^{\frac{\alpha_g+\alpha_f}{2}} |\widehat{\widetilde{\psi} g}(\eta)| \dd \eta \Bigg )^r \dd \xi \Bigg )^{1/r}.
    \end{split}
    \]
    Note that we have used $\alpha_f+\alpha_g \geq 0$ zero here. Finally, we apply \eqref{eq:xietaest} as well as Young's inequality to bound this term. This completes the proof of claim \eqref{item:1.7.2}.

    Now, let us assume the hypotheses of claim \eqref{item:1.7.3} and suppose $\epsilon$ is chosen small enough so that $C_1 \subset \WF_{\vx_0}(f)^c$. Our objective in this case is to show that \eqref{eq:tHarest} is finite when $\alpha = \alpha_g + \min\{\alpha_f,0\}$. As before, we use \eqref{eq:partition} to split the integral into four different terms and bound these individually. The first term can be bounded using \eqref{eq:xiest1} because $\xi-\eta \notin WF_{\vx_0}(f)$ and $\widetilde{\psi} g \in H^{\alpha_g,q}(\mathbb{R}^n)$. The second term can be bounded by \eqref{eq:2ndest} with $\alpha_f + \alpha_g$ replaced by $\alpha_g$ and this is then bounded by Young's inequality since $\xi-\eta \notin WF_{\vx_0}(f)$ and $\alpha_g \geq 0$.

    The third and fourth terms are bounded in the same way using the fact that $\xi \in C_1$ and $\xi-\eta \in C_2$. Because of this, we combine the two terms and eliminate the cut-off functions depending on $\eta$. First, consider the case when $\alpha_f \geq 0$. Since $\xi-\eta \in C_2$, we can use \ref{eq:conelem2} to bound by 
    \[
    C \left (\int_{\mathbb{R}^n} \left ( \int_{\mathbb{R}^n} \chi_{C_2}(\xi-\eta) u(\xi) |\widehat{\psi f}(\xi-\eta)| (1+|\eta|^2)^{\frac{\alpha_g}{2}}|\widehat{\widetilde{\psi} g}(\eta)| \dd \eta \right )^r \dd \xi \right )^{1/r}.
    \]
    This is bounded by Young's inequality. Now suppose $\alpha_f<0$ and multiply and divide by $(1+|\xi-\eta|^2)^{\alpha_f/2}$ to bound by
    \[
    \begin{split}
    & C \Bigg (\int_{\mathbb{R}^n} \Bigg ( \int_{\mathbb{R}^n} \chi_{C_2}(\xi-\eta) u(\xi) (1+|\xi|^2)^{\frac{\alpha_g + \alpha_f}{2}} (1+|\xi-\eta|^2)^{-\frac{\alpha_f}{2}}\\
    & \hskip15em (1+|\xi-\eta|^2)^{\frac{\alpha_f}{2}}|\widehat{\psi f}(\xi-\eta)||\widehat{\widetilde{\psi} g}(\eta)| \dd \eta \Bigg )^r \dd \xi \Bigg )^{1/r}.
    \end{split}
    \]
    Finally, we use \eqref{eq:xietaest}, \eqref{eq:conelem2}, and Young's inequality to bound this term. This completes the proof of claim \eqref{item:1.7.3}.
    
\end{proof}

We are now in a position to prove our main microlocal theorems.

\begin{theorem}\label{main_thm_1}
Let $f = u\chi_{\Omega}$, where $\Omega \subset B$ is a compact domain with smooth boundary, and $u \in C^\infty(\mathbb{R}^2)$. Let $\varphi \in H^{\lambda,\infty}(\mathbb{R}^2)$ for $\lambda > 1/2$, and let $g = e^{ -\widetilde{\mathcal{V}_L}f(\cdot,\phi) }$ where $\phi \in [-\pi, \pi]$ is fixed. If $(\vx_0,\xi) \notin WF(f)$ then $fg \in H^{\beta,2}$ microlocally near $(\vx_0,\xi)$ for any $\beta < \lambda$. If we also assume $u>0$, then
\begin{equation}
\WF(fg)\backslash \WF^{1/2}(fg) = \WF(f).
\end{equation}
\end{theorem}

%\seanc{Is the simply connected hypothesis necessary? Change to compact domain with smooth here, also perhaps above.}

\begin{proof}

%First we note that when $\vx_0 \notin \text{ssupp}(f)$, $fg$ is in $H^{1,2}$ near $\vx_0$ since $f$ is smooth near $\vx_0$ and $g$ is in $H^{1,2}$ near $\vx_0$ by Theorem \ref{thm_sep} and since $\lambda > 1/2$. Therefore $\{(\vx_0,\xi) \ : \ \xi\neq 0 \} \subset (\WF(fg)\setminus \WF^{1/2}(fg))^c$.

%Now we consider the case when $\vx_0 \in \text{ssupp}(f)$.
Let $\vx_0 \in \mathbb{R}^2$ and $\psi \in C^\infty_c(\mathbb{R}^2)$ be a smooth cutoff centered on $\vx_0$. Then $\psi g$ is in $H^{\alpha' + \lambda,2}(\mathbb{R}^2)$ %  and $H^{\alpha' + \lambda-1,1}(\mathbb{R}^2)$, 
where $\alpha' < 1/2$ by Theorem \ref{thm_sep}, and $f\in H^{\alpha',2}(\mathbb{R}^2)$ which implies by Lemma \ref{lem:Hsa_cont} with $p = 1$ and $\widetilde{p} = 2$ that $f \in H^{\alpha'-1,1}(\mathbb{R}^2)$. Now let us set $q = r = 2$ and $p = 1$ and apply Theorem \ref{mainprod}. If $0 \neq\xi_0 \notin WF_{\vx_0}(f)$, then claim \eqref{item:1.7.3} of Theorem \ref{mainprod} yields $fg \in H^{\beta,2}$ microlocally near $(\vx_0,\xi_0)$, for any $\beta < \lambda$ which proves first claim of the theorem. Since $\lambda > 1/2$ this implies that $\{(\vx_0,\xi_0)\} \in (\WF(fg)\setminus \WF^{1/2}(fg))^c$. Combining this with the previous paragraph, we can conclude $WF(f) \supset \WF(fg)\setminus \WF^{1/2}(fg)$.  %When $.\xi_0 \notin WF_{\vx_0}(f)$, $fg \in H^{\alpha,2}$ near $(\vx_0,\xi_0)$ by (3) of Theorem \ref{mainprod}. 

%\seanc{I think I have fixed this now.}

%\seanc{I think there is a problem here which needs to be fixed. If we directly apply Theorem 4.8 with $p = r = 2$ and $q =1$, then we use part (3) of the theorem with $f\in H^{\alpha',2}(\mathbb{R}^2)$ and $\psi g \in H^{\alpha'+\lambda -1,1}(\mathbb{R}^2)$. I think this gives $\psi f g \in H^{\alpha' + \lambda -1 ,2}(\mathbb{R}^2)$ (\tred{I think this should be $2\alpha'$ in the order, which would be OK. If I've still made a mistake, then note we can just set $\lambda > 1$ as this is what we ultimately assume anyway. Hope this helps and so you maybe don't have to do too much fixing.}) which is not what we want. I think we actually need to apply Theorem 4.8 part (3) with $q = r = 2$ and $p =1$, $f \in H^{\alpha'-1,1}(\mathbb{R}^2)$ and $\varphi g \in H^{\alpha' + \lambda,2} (\mathbb{R}^2)$. However, I don't think part (3) actually holds with $\alpha_f < 0$ in general. But I think we can modify part (3) to include some cases of $\alpha_f<0$ and that should fix the problem. I am working on this.}

Now let us assume that $u>0$ and focus on the case $\xi_0 \in WF_{\vx_0}(f)$. Let $\psi = 1$ on a ball, $B'$, centered on $\vx_0$. First, by Definition \ref{sobo_spaces}, for $0<\alpha'<1$, $\|\psi fg\|_{H^{\alpha',2}(\mathbb{R}^2)}$ converges if and only if the seminorm $|\psi fg|_{H^{\alpha'}(\mathbb{R}^2)}$ converges. We have
\begin{equation}
\begin{split}
|\psi fg|^2_{H^{\alpha}(\mathbb{R}^2)} &= \int_\Omega\int_{\Omega^c} \frac{|(\psi fg)(\vy)|^2}{|\vx-\vy|^{2 + 2\alpha'}}\mathrm{d}\vx\mathrm{d}\vy  + \int_{\Omega^c}\int_{\Omega} \frac{|(\psi fg)(\vx)|^2}{|\vx-\vy|^{2 + 2\alpha'}}\mathrm{d}\vx\mathrm{d}\vy + |\psi fg|^2_{H^{\alpha'}(\Omega)} \\
&\geq \int_{\Omega\cap B'}\int_{\Omega^c\cap B'} \frac{|(fg)(\vy)|^2}{|\vx-\vy|^{2 + 2\alpha'}}\mathrm{d}\vx\mathrm{d}\vy + \int_{\Omega^c\cap B'}\int_{\Omega\cap B'} \frac{|(fg)(\vx)|^2}{|\vx-\vy|^{2 + 2\alpha'}}\mathrm{d}\vx\mathrm{d}\vy\\
&\geq C^2 \paren{ \int_{\Omega\cap B'}\int_{\Omega^c\cap B'} \frac{1}{|\vx-\vy|^{2 + 2\alpha'}}\mathrm{d}\vx\mathrm{d}\vy + \int_{\Omega^c\cap B'}\int_{\Omega\cap B'} \frac{1}{|\vx-\vy|^{2 + 2\alpha'}}\mathrm{d}\vx\mathrm{d}\vy }\\
&= C^2 |\chi_{\Omega}|^2_{H^{\alpha'}(B')},
\end{split}
\end{equation}
where $fg > C$ on $\Omega$. Thus, $\|\psi fg\|_{H^{1/2,2}}$ diverges as $\chi_\Omega$ is not in $H^{1/2,2}(B')$. 
%\jc{\tred{If we need more rigour here could cite \url{https://math.stackexchange.com/questions/3825303/proof-that-the-characteristic-function-of-a-bounded-open-set-is-in-h-alpha}}}
%\seanc{I don't think we need this.}

%\seanc{I think there was a slight error below since, as previously written, $\mathbb{R}^2 \setminus V$ contained a direction in the wavefront set ($-\xi_0$). I think this is a minor technical problem and have tried to fix it. Please review.}
Let $V = \{\xi : |\xi\cdot \xi_0| > |\xi||\xi_0|\cos\epsilon\}$ be an open cone interior with central axis $\xi_0 \in WF_{\vx_0}(f)$. Note that, since $\partial \Omega$ is smooth, $\WF_{\vx_0}(f) = \mathrm{span}(\{\xi_0\})\setminus \{0\}$ and so $(\mathbb{R}^2 \setminus V) \cap \WF_{\vx_0}(f) = \emptyset$ Then 
\begin{equation}\label{equ_loc}
\int_V |\mathcal{F}(\psi fg)(\xi)|^2(1+|\xi|^2)^{1/2} \mathrm{d}\xi= \|\psi fg\|^2_{H^{1/2,2}} - \int_{\mathbb{R}^2\backslash V} |\mathcal{F}(\psi fg)(\xi)|^2(1+|\xi|^2)^{1/2}\mathrm{d}\xi.
\end{equation}
The third term in \eqref{equ_loc} converges for any $\epsilon > 0$ since $fg$ is in $H^{1/2,2}$ near $(\vx_0, \xi)$ for any $\xi \in \mathbb{R}^2\backslash V$ as shown earlier in this proof. Thus, the term on the left side of \eqref{equ_loc} diverges no matter how small we choose $\epsilon$ since $\|\psi fg\|_{H^{1/2,2}}$ diverges. Therefore, putting this together, when $\xi_0 \in WF_{\vx_0}(f)$, $fg$ is not in $H^{1/2,2}$ near $(\vx_0,\xi_0)$. This implies $WF(f) \subset \WF(fg)\setminus \WF^{1/2}(fg)$ which completes the proof.
\end{proof}

%\seanc{Changing the end of this proof slightly as I'm not sure we've proved the if and only if statement which was there previously.}

\begin{remark}
The above theorem means that $fg$ has two distinct sets of singularities, those which are more regular (i.e., $H^{1/2,2}$ microlocally), and those which are more irregular (i.e., not in $H^{1/2,2}$ microlocally). The latter correspond to the singularities of $f$, which are one of the targets for reconstruction. Noting that $h(\cdot,\theta) = fg = f e^{ -\widetilde{\mathcal{V}_L}f(\cdot,\theta) } = f e^{ -\widetilde{\mathcal{V}_{a,b}}f(\cdot,\theta) }$, $\mathcal{R}f(\cdot,\theta) = R_w(h(\cdot, \theta))$ is simply the 1-D Radon projection projection (in direction $\theta$) of the 2-D slice of $h$ on the $x_3 = \theta$ plane. When the ray $L(s,\theta)$ is tangent to $\partial \Omega$, we would expect to see a stronger singularity in the non-linear data $\mathcal{R}f$, and conversely when $L(s,\theta)$ is not tangent to $\partial \Omega$. {To help visualize this idea, see figure \ref{Fproj}. Here we have shown an example $f$, a slice of $h' = wfg$, where $w$ is the physical modeling weight, and the corresponding Radon projection of $h'$. The physical parameters used here (e.g., the V-line opening angle $2\psi$) are the same as in section \ref{simulations}, so refer to that section for more details. We have also highlighted the projection line and three point-direction pairs $(\vx_i,\xi_i)$, $i = 1,2,3$, in figure \ref{Fp_b}. Here $(\vx_1,\xi_1) \in \WF(f)$, where a jump singularity occurs. This edge is not in $H^{1/2,2}$ locally and this is highlighted in the projection. At $(\vx_2,\xi_2)$, $h'$ is smooth, and hence we do not see an edge in the data. There could also be singularities of $h'$ disjoint from $\WF(f)$ at, e.g., $(\vx_3,\xi_3)$. We show in Theorem \ref{main_thm_1} that these edges are in $H^{\beta,2}$ locally for any $\beta < \lambda$, where $\varphi \in H^{\lambda,\infty}(\mathbb{R}^2)$. Thus, if $\lambda$ is large enough, the singularities of $f$ are uniquely encoded in $h$, and this is reflected in the singularities of the projection. We aim to formalize this idea in the following analysis.

It is important to note also that, due to $w$ and attenuation, the singularities on the upper-right side of the $h'$ slice are of less magnitude than those towards the bottom-left of the phantom, although the Sobolev order is unchanged. This can have practical implications on edge recovery which we will explore later in section \ref{simulations}.}
\begin{figure}
\centering
\begin{subfigure}{0.3\textwidth}
\includegraphics[width=0.9\linewidth, height=3.2cm, keepaspectratio]{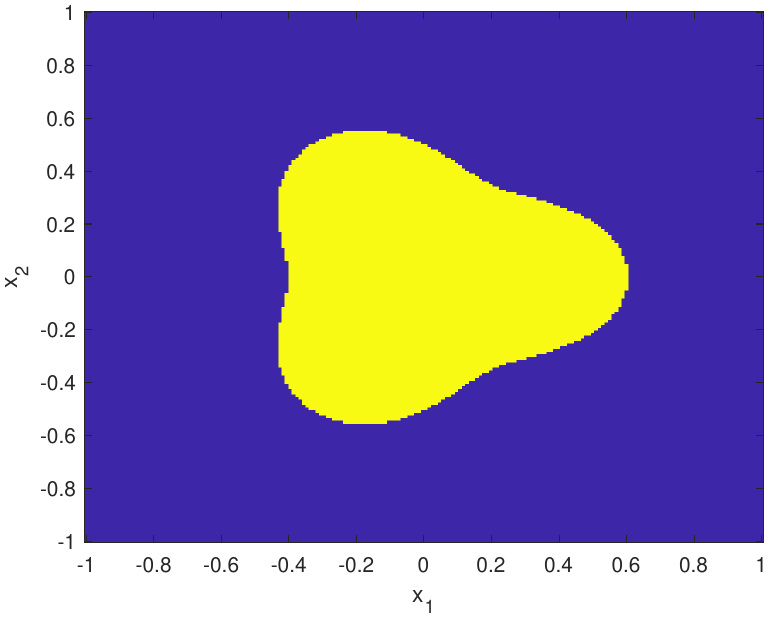}
\subcaption{$f$}
\end{subfigure}
\begin{subfigure}{0.3\textwidth}
\includegraphics[width=0.9\linewidth, height=3.2cm, keepaspectratio]{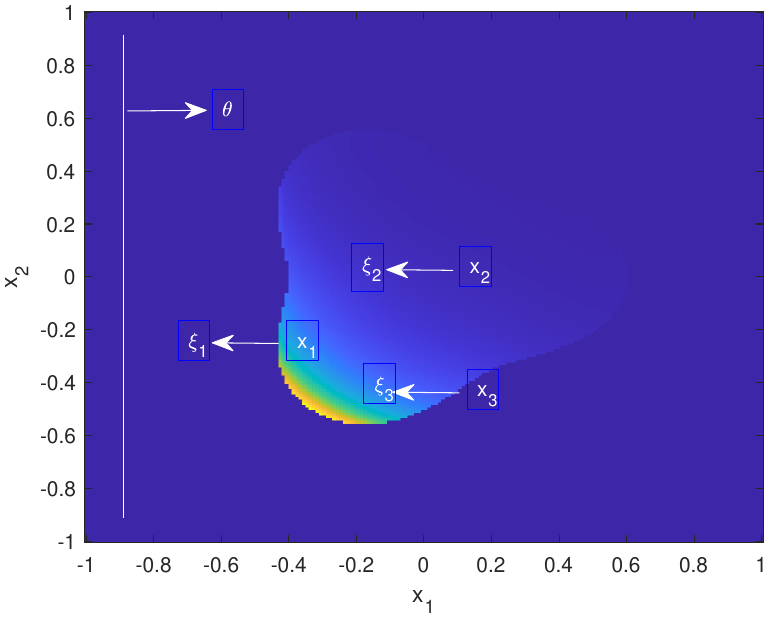}
\subcaption{$h'(\cdot,\pi/2)$} \label{Fp_b}
\end{subfigure}
\begin{subfigure}{0.3\textwidth}
\includegraphics[width=0.9\linewidth, height=3.2cm, keepaspectratio]{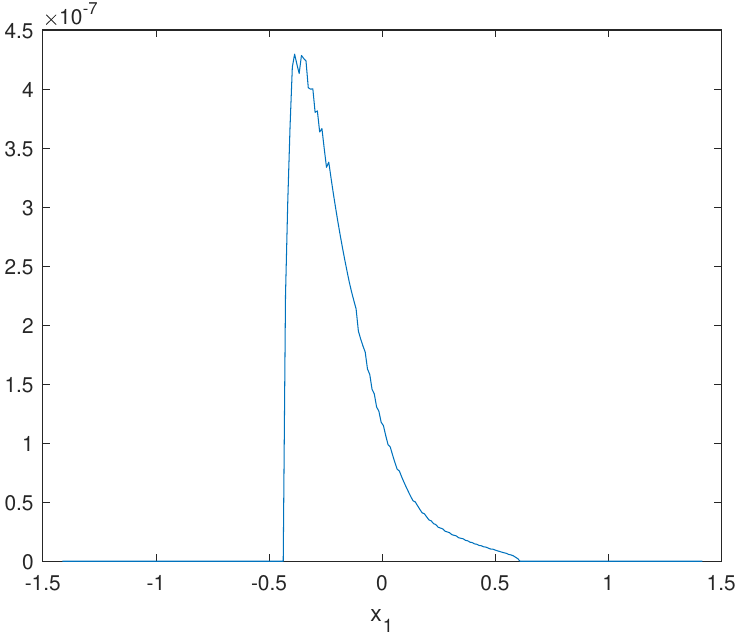}
\subcaption{$R_w(h'(\cdot,\pi/2))(\cdot,\pi/2)$}
\end{subfigure}
\caption{(a) - Example density. (b) - slice of $h' = wfg$ when $\phi = \pi/2$, where $w > 0$ is a smooth weight included for physical modeling as discussed in section \ref{model}. (c) - Radon projection of $h'$ at $\theta = \pi/2$. By \eqref{R_to_R}, this is the same as the projection of the non-linear data at $\theta = \pi/2$.}
\label{Fproj}
\end{figure}
\end{remark}
  
\begin{theorem}
\label{main_1}
Let $f = u\chi_\Omega$ be as in Theorem \ref{main_thm_1}, and let $(s_0,\theta_0)$ be such that $L = L(s_0,\theta_0)$ is not tangent to $\partial \Omega$. Let $\varphi \in H^{\lambda,\infty}(\mathbb{R}^2)$ for $\lambda > 1/2$. Then $\mathcal{R}f(\cdot,\theta_0)$ is in $H^{\beta,2}$ near $s_0$ for any $\beta < \lambda$.
\end{theorem}
\begin{proof}
Fixing $\phi = \theta_0$, we have $\mathcal{R}f(s, \theta_0)= R_w(h(\cdot, \theta_0))(s,\theta_0)$. Let $h = h(\cdot,\theta_0)$ for convenience of notation. By Theorem \ref{main_thm_1}, since $L$ is not tangent to $\partial \Omega$, $h$ is in $H^{\beta,2}$ microlocally near $(\vx, \theta_0)$ for any $\vx \in L$. Thus, $R_wh \in H^{\beta + 1/2,2}$ near $(s_0,\theta_0)$ \cite[Theorem 3.1]{Q1993sing}. Let $\psi(s,\theta) = \psi_1(s)\psi_2(\theta)$ be a smooth cutoff centered on $(s_0,\theta_0)$ where $\psi_1(s_0),\psi_2(\theta_0) = 1$ and $\psi R_wh \in H^{\beta + 1/2,2}(\mathbb{R}^2)$. Applying Theorem \ref{trace_thm}, we have $\psi_1 R_wh(\cdot,\theta_0) = \psi_1 \mathcal{R}f(\cdot, \theta_0) \in H^{\beta,2}(\mathbb{R}^2)$ for any $\beta < \lambda$.
\end{proof}

\begin{theorem}
\label{main_2}
Let $f = u\chi_\Omega$ be as in Theorem \ref{main_thm_1} with $u>0$ and let $(s_0,\theta_0)$ be such that $L = L(s_0,\theta_0)$ is tangent to $\partial \Omega$ and the set of tangent points is either a line segment or a single point at which the curvature of $\partial \Omega$ is not zero. Also suppose that $\varphi \in H^{\lambda,\infty}(\mathbb{R}^2)$ with $\lambda > 1$. Then, $\mathcal{R}f(\cdot,\theta_0)$ is not in $H^{1,2}$ near $s_0$.
\end{theorem}
\begin{proof}
%\seanc{Working on this Theorem and proof. We will need to change the statement because it is not true in general.}

Let $T \subset \mathbb{R}^2$ be the set of points where $L(s_0,\theta_0)$ is tangent to $\partial \Omega$. By hypothesis, $T$ is either a single point or a line segment and in either case let $\psi \in C_c^\infty(\mathbb{R}^2)$ be a cut-off function which is equal to $1$ in a neighbourhood of $T$. Then we have
\begin{equation}\label{eq:Rdecomp}
\mathcal{R}f(s,\theta_0) = \lambda \int_{L(s,\theta_0)}w \psi f e^{-\widetilde{\mathcal{V}_{L}}f(\vx,\theta_0)} \dd l + \lambda \int_{L(s,\theta_0)}w (1-\psi) f e^{-\widetilde{\mathcal{V}_{L}}f(\vx,\theta_0)} \dd l
\end{equation}
A slight variation of the proofs of Theorems \ref{main_thm_1} and \ref{main_1} shows that the second term on the right side of \eqref{eq:Rdecomp} will be in $H^{\beta,2}$ near $s_0$ for $\beta < \lambda$. Since $\lambda > 1$, it is therefore sufficient to show that the first term on the right side of \eqref{eq:Rdecomp} is not in $H^{1,2}$ near $s_0$. For notational convenience, let us write
\[
R(s) = \lambda \int_{L(s,\theta_0)}w \psi f e^{-\widetilde{\mathcal{V}_{L}}f(\vx,\theta_0)} \dd l
\]
for this term.

Before continuing, note that since $\psi f \in H^{\alpha,2}(\mathbb{R}^2)$ for any $\alpha< 1/2$ and since $\lambda > 1$, by Theorem \ref{thm_V_smth}, $\widetilde{V_L} f(\cdot,\theta_0) \in H^{3/2+,2}(\mathbb{R}^2)$ where the plus sign indicates membership in the intersection of spaces with order higher than $3/2$. Furthermore, all derivatives of $\widetilde{V_L} f(\cdot,\theta_0)$ will be in $H^{1/2+,2}(\mathbb{R}^2)$. Sobolev embedding also implies that $\widetilde{V_L} f(\cdot,\theta_0)$ is continuous.

Now, let us consider the case when $T$ is a line segment. Without loss of generality, we suppose that $\Theta_0 = \Theta(\theta_0) = (\cos(\theta),\sin(\theta))^T$ points into $\Omega$ on $T$. By continuity of the integrand,
\[
\lim_{s\rightarrow s_0^+} R(s) - \lim_{s\rightarrow s_0^-} R(s) = \int_T u e^{-\widetilde{\mathcal{V}_{L}}f(\vx,\theta_0)} \dd l > 0.
\]
Therefore $Rf(\cdot,\theta_0)$ is not continuous at $s_0$ and so cannot be in $H^{1,2}$ near $s_0$ by Sobolev embedding.

Now consider the case when $T$ is a single point $\vx_0 = s_0 \Theta_0 + t_0 (\Theta_0)_\perp$ where the curvature of $\partial \Omega$ does not vanish. Note that
\[
\lambda \int_{L(s,\theta_0)} w \psi u e^{-\widetilde{\mathcal{V}_{L}}f(\vx,\theta_0)} \dd l = R(s) + \lambda \int_{L(s,\theta_0)} w \psi u \chi_{\Omega^c} e^{-\widetilde{\mathcal{V}_{L}}f(\vx,\theta_0)} \dd l.
\]
Since $w \psi u \in C_c^\infty(\mathbb{R}^2)$ and $e^{-\widetilde{\mathcal{V}_{L}}f(\vx,\theta_0)} \in H^{3/2+,2}(\mathbb{R}^2)$, the term on the left side of this equation is in $H^{1,2}(\mathbb{R})$. Therefore, to show $R$ is not in $H^{1,2}(\mathbb{R})$ it is sufficient to show that either term on the right side is not in $H^{1,2}(\mathbb{R})$. Thus, by possibly replacing $\Omega$ by $\Omega^c$, we can assume without loss of generality that, at $\vx_0$, $\Theta_0$ points into $\Omega$ and when $\partial \Omega$ is oriented by $\Theta_0$ its curvature is positive.

Next, taking the support of $\psi$ to be sufficiently small, for $s<s_0$ we have $R(s) = 0$ while for $s>s_0$, since the curvature is positive, there will be exactly two values $t_\pm(s)$ such that $\vx_0 + s \Theta_0 + t_\pm(s) (\Theta_0)_\perp \in \partial \Omega$ within the support of $\psi$. We suppose $\pm t_\pm(s)>0$. For $s>s_0$ in the support of $\psi$, we then have
\[
R(s) = \lambda \int_{t_-(s)}^{t_+(s)} wu(\vx_0 + s  \Theta_0 + t (\Theta_0)_\perp) e^{-\widetilde{\mathcal{V}_{L}}f(\vx_0 + s  \Theta_0 + t (\Theta_0)_\perp,\theta_0)}\dd t.
\]
Let us write $H(s,t)$ for the integrand in the previous formula. Then $H \in H^{3/2+,2}(\mathbb{R}^2)$ and $\frac{\partial H}{\partial s} \in H^{1/2+,2}(\mathbb{R}^2)$ and by Theorem \ref{trace_thm} we can restrict to the segment of integration and the result is in $H^{0+,2}$ on this segment. Thus, for $s$ larger than $s_0$, we have
\[
R'(s) = t_+'(s) H(s,t_+(s)) - t'_-(s) H(s,t_-(s)) + G(s,t)
\]
where $G$ is bounded. Because the curvature of $\partial \Omega$ is positive at $\vx_0$, $t_\pm(s) \rightarrow 0$ as $s \rightarrow s_0^+$ and
\[
\lim_{s\rightarrow s_0^+} (s-s_0)^{1/2} t'_{\pm}(s) = \pm \kappa
\]
for some $\kappa > 0$. Since $H$ is continuous and $(s-s_0)^{1/2}$ is not in $H^{1,2}(\mathbb{R})$ near $s = s_0$, therefore $R$ is not in $H^{1,2}(\mathbb{R})$ near $s_0$ which completes the proof.
\end{proof}

We now have our main result regarding the singularities of $\mathcal{R}f$.
\begin{corollary}
\label{corr_main}
Let $f = u\chi_\Omega$ be as in Theorem \ref{main_thm_1} with $u>0$ and let $\varphi \in H^{\lambda,\infty}(\mathbb{R}^2)$ with $\lambda>1$. Let
\[
\begin{split}
\Pi_\theta & = \Big \{s\in \mathbb{R} \ : \ L(s,\theta) \mbox{ is tangent to $\partial \Omega$ exactly in one line segment or at }\\
&\hskip2cm \mbox{a single where the curvature of $\partial \Omega$ is not zero}\Big \}.
\end{split}
\]
Then,
\begin{equation}
\Pi_{\theta_0} \subset
\text{ssupp}(\mathcal{R}f(\cdot,\theta_0)) \backslash \text{ssupp}^1(\mathcal{R}f(\cdot,\theta_0)) \subset \text{ssupp}(R_wf)\cap \{\theta = \theta_0\}
\end{equation}
almost everywhere.
\end{corollary}
\begin{proof}
The result follows after applying Theorem \ref{main_1} with some $\lambda > 1$, and also using Theorem \ref{main_2}.
\end{proof}

Thus, for $\varphi$ smooth enough, the strongest (most irregular) singularities of the non-linear transform $\mathcal{R}f$ correspond to those of its linear counterpart, $R_w f$, except possibly at lines of multiple tangency. For example, if $\varphi  = \chi_{B'}$ is a characteristic function on a ball (assuming a spherical pixel), then $\hat{\varphi}$ decays at rate $3/2$ \cite{f1,f2} and therefore $\varphi \in H^{3/2,\infty}(\mathbb{R}^2)$, which is sufficient to apply the above results. Later, in section \ref{simulations}, we show simulated image reconstructions when $\varphi = \chi_{B'}$ to validate our theory.

If $\partial \Omega$ has non-zero curvature everywhere (for example $\Omega$ is strictly convex), then, as observed in the discussion after Theorem \ref{V_smooth}, $\widetilde{\mathcal{V}_L} f(\cdot,\phi) \in H^{\alpha,2}(\mathbb{R}^2)$ for any $\alpha < 1$. In this case, the hypotheses of Theorem \ref{main_2} and Corollary \ref{corr_main} can be weakened to $\lambda > 1/2$.

If the kernel $\varphi$ is not sufficiently regular, then it is possible for $\widetilde{\mathcal{V}_L}f$ to contribute singularties to $\mathcal{R}f$ which are the same Sobolev order as those which arise from $f$. Thus, it is not possible based only on the order of the singularities in $\mathcal{R}f$ to distinguish those which are normal to $\partial \Omega$ (i.e. in the wavefront set of $f$) from those which arise from $\widetilde{V_L}f$.

%\seanc{Are you making a numerical example illustrating this point?}

We now apply the above results to determine $\partial \Omega$. %To make the result global in the next corollary, we require one additional hypothesis which ensures there are only finitely many lines tangent to $\partial \Omega$ on a disconnected set.

\begin{corollary}
\label{corr_main_2}
Let $f = u\chi_\Omega$ be as in Theorem \ref{main_thm_1} and let $\varphi \in H^{\lambda,\infty}(\mathbb{R}^2)$ with $\lambda > 1$. Further, suppose that the set $\Upsilon$ of $(s,\theta)$ such that $L(s,\theta)$ is tangent to $\partial \Omega$ on a disconnected set or at a single point with zero curvature is finite. %Further, let $\partial \Omega$ have non-zero curvature.
Then $\mathcal{R}f$ determines $\partial \Omega$.
\end{corollary}
\begin{proof}
First note that $\partial \Omega = \text{ssup}(f)$, and so we will prove that this singular support can be determined by $\mathcal{R}f$. By Corollary \ref{corr_main},
\[
\text{ssupp}(R_wf) \setminus \Upsilon\subset\bigcup_{\theta_0} \text{ssupp}(\mathcal{R}f(\cdot,\theta_0)) \backslash \text{ssupp}^1(\mathcal{R}f(\cdot,\theta_0)) \subset \text{ssupp}(R_wf).
\]
Since $\Upsilon$ is finite, $\text{ssup}(R_w f)$ contains no discrete points, and $\text{ssup}(R_w f)$ is closed, by taking the closure of the union in the previous displayed formula we can determine $\text{ssup}(R_wf)$ from $\mathcal{R}f$.

Let $\Upsilon'$ be the set of $(s,\theta)$ which are in $\Upsilon$ or are such that $L(s,\theta)$ is tangent to $\partial \Omega$ on a line segment. Then, $\text{ssup}(R_wf)\setminus \Upsilon'$ is a finite union of smooth curves, $\cup_{i=1}^k \gamma_i$, corresponding to the components of $\partial \Omega$ with points where the curvature is zero removed; each of these components gives rise to two of these curves. We point out that the closures of these curves will intersect at points in the sinogram space corresponding to lines tangent to $\partial \Omega$ at multiple points, and may not be smooth at points coresponding to lines which intersect $\partial \Omega$ at points of zero curvature.

Now consider the normal bundles of the curves $\gamma_i$, which will be in the wavefront set of $R_w f$. Applying the inverse of the canonical relation of $R_w$ to these normal bundles, we obtain the wavefront set of $\chi_\Omega$ with the points at which the curvature of $\partial \Omega$ vanishes or there is a tangent line simultaneously tangent to $\partial \Omega$ at another place removed. Projecting this set to $\mathbb{R}^2$ and taking its closure, we obtain $\partial \Omega$ with any flat segments removed. This set is composed of some finite number of smooth curves $\{\alpha_j\}_{j=1}^l$.

Now consider $(s,\theta)$ such that $L(s,\theta)$ is tangent to $\partial \Omega$ along some number of flat line segments. Since $\partial \Omega$ is smooth, there will be a number of the curves $\alpha_j$ which approach tangency to $L(s,\theta)$ at one, or possibly both, or their endpoints. Pairing up such endpoints which are adjacent along $L(s,\theta)$, we can then connect the endpoints to obtain the full set $\partial \Omega$. This completes the proof.

% Note that these curves may intersect at points corresponding to lines which are tangent to $\partial \Omega$ at more than one point. The curves will be smooth away from points which are in $\Upsilon$ or which correspond to flat sections of $\partial \Omega$.

% \seanc{I am still working on this.}

% $\text{ssupp}(R_wf)$ is uniquely determined a.e. by $\mathcal{R}f$. Since $\partial \Omega$ has non-zero curvature, $\text{ssupp}(R_wf)$ is the finite union of smooth curves, $\cup_{i = 1}^k \gamma_i$, in sinogram space, where each $\gamma_i \subset \mathbb{R} \times [-\pi,\pi]$ is a 1-D smooth curve. The $\gamma_i$ intersect if and only if a ray $L(s,\theta)$ is tangent to two boundary points simultaneously, which only happens at a finite number of points. Otherwise, the directions of the singularities of $R_w f$ are simply the normal vectors to the $\gamma_i$. Thus, $\text{ssupp}(R_wf)$ in this case determines $\WF(R_w f)$ a.e., and hence $\WF(f)$ a.e.
\end{proof}
% \begin{remark}
% By Corollary \ref{corr_main}, the only parts of $\text{ssupp}(R_wf)$ which are not necessarily determined by $\mathcal{R}f$ are isolated points, and these correspond to cases when the $\gamma_i$ intersect. When $\partial \Omega$ is convex, the $\gamma_i$ are disjoint and we can determine $\WF(f)$ everywhere. 

Based on the above theory, and applying ideas from lambda tomography \cite{quinto2009electron}, we would expect to see the singularities of $f$ highlighted in a reconstruction of the form
\begin{equation}
f_r = R^* \frac{\mathrm{d}^k}{\mathrm{d}s^k} \mathcal{R}f,
\end{equation}
where $R = R_1$ (i.e., the classical Radon transform), $k \geq 1$ and $f_r$ denotes a reconstruction of $f$.
%\end{remark}

\subsection{Uniqueness results for the density value}
So far, we have derived uniqueness results for the singularities of $f$. We now focus on the smooth part of $f$ (namely $u$) in the case when $u = n_e$ is a constant density value.

\begin{theorem}
\label{ne_thm}
Let $f = u\chi_\Omega$ be as in Theorem \ref{main_thm_1} and let $u = n_e$ be constant, where $n_e \in [0, u_m]$ and $u_m$ is a maximum possible density value. Let $\varphi = \frac{1}{\epsilon^2}\chi_{B_\epsilon(0)}$ for some $\epsilon > 0$. Then, if $\epsilon$ is small enough, $\mathcal{R}f$ uniquely determines $n_e$.
\end{theorem}

%\seanc{I'm not sure that it's necessary to include $\epsilon$ in this Theorem/proof at all. I think $\widetilde{V_L} \chi_\Omega \rightarrow 0$ as $\delta \rightarrow 0^+$ for fixed $\epsilon$ which is sufficient for the proof. This also affects the next corollary.}

\begin{proof}
We have
\begin{equation}
\mathcal{R}f(s,\theta) = \lambda n_e\int_{L(s,\theta)}w \chi_\Omega(\vx) e^{-n_e \widetilde{ \mathcal{V}_L }\chi_\Omega(\vx, \theta)}\mathrm{d}l.
\end{equation}
Since $\partial \Omega$ is smooth, there exits a $(s_0,\theta_0)$ such that $L(s_0,\theta_0) \cap \partial \Omega = \{\vx_0\}$ where $\partial \Omega$ has non-zero curvature at $\vx_0$ and $f(\vx) = 0$ for $\vx \cdot \Theta_0 \geq s_0$. Let us fix $\theta = \theta_0$ and consider $s\in [s_0-\delta,s_0]$. Then,
\begin{equation}
\label{equ_der}
\frac{\mathrm{d}}{\mathrm{d}n_e}\mathcal{R}f(s,\theta) = \lambda \int_{L(s,\theta)}w \chi_\Omega\left[1 - n_e\widetilde{ \mathcal{V}_L } \chi_\Omega\right] e^{-n_e \widetilde{ \mathcal{V}_L }\chi_\Omega}\mathrm{d}l.
\end{equation}
By construction, the term $\widetilde{ \mathcal{V}_L }\chi_\Omega$ in the square bracket in the above integral goes to zero as $\epsilon, \delta \to 0$. Thus, for $\epsilon$ and $\delta$ small enough, $1 - n_e \widetilde{ \mathcal{V}_L }\chi_\Omega > 0$ for $n_e \leq u_m$. Therefore, $\mathcal{R}f(s,\theta_0)$ for any fixed $s \in [s_0-\delta,s_0]$, is a strictly monotone increasing function of $n_e$ (noting that the remaining terms in \eqref{equ_der} are strictly positive), and these data values thus correspond uniquely to $n_e$.
\end{proof}

\begin{corollary}
Let $f = u\chi_\Omega$ be as in Corollary \ref{corr_main_2} with $u = n_e$ be constant. Let $\varphi = \frac{1}{\epsilon^2}\chi_{B_\epsilon(0)}$ with $\epsilon > 0$ small enough. Then, $\mathcal{R}f$ determines $f$ uniquely.
\end{corollary}
\begin{proof}
By Corollary \ref{corr_main_2}, and since $\hat{\varphi}$ decays at rate $3/2 > 1$, $\mathcal{R} f$ determines $\partial \Omega$ uniquely. By Theorem \ref{ne_thm}, $u = n_e$ is also determined. We now argue that $\partial \Omega$ determines $\Omega$. Indeed, take any $x \in \mathbb{R}^2\setminus \partial \Omega$ and let $L_x$ be an infinite ray with $x$ as the origin which is never tangent to $\partial \Omega$. Let $n$ be the number of points in $\partial \Omega \cap L_r$.  Since $\Omega$ is compact, points sufficiently far along this ray will be outside of $\Omega$ and from this we can conclude that if $n$ is odd $x \in \Omega$ and if $n$ is even then $x \notin \Omega$. Thus, $\Omega$ is also determined and this completes the proof.
\end{proof}

%\seanc{I don't think simply connected is required for $\partial \Omega$ to determine $\Omega$. I believe compact with smooth boundary is sufficient. Do you agree? Add argument for recovery of $\Omega$ from $\partial \Omega$ for non-simply connected set.}

\section{Simulated image reconstructions}
\label{simulations}
In this section, we present simulated image reconstructions to verify our theory. We first show edge reconstructions and then give an example where we determine the density value.
\subsection{Data simulation}
We simulate data using the non-linear model in \eqref{model_1} for $s\in [-\sqrt{2}, \sqrt{2}]$ (noting that $\text{supp}(f)\subset B$) and $\theta \in [0,2\pi]$. The image resolution is $N \times N$ with $N = 200$ and $(s,\theta)$ are sampled uniformly on $[-\sqrt{2}, \sqrt{2}] \times [0,2\pi]$ on a $282 \times 360$ grid, sampling 282 points in $s$ and 360 points in $\theta$, respectfully. Letting $\vb$ be a vector of data samples, we simulate noisy data \begin{equation}
\vb_\gamma = \vb + \gamma \times \frac{\|\vb\|_2}{\sqrt{k}}\eta,
\end{equation}
where $k$ is the length of $b$, $\eta \sim \mathcal{N}(0,1)$ is a vector of draws from the standard normal distribution, and $\gamma$ controls the noise level. We set $\varphi = \chi_{B_{0.02}(0)}$ (i.e., the smoothing kernel for the V-line transform) to be a characteristic function on a ball with radius 0.02, which is the length of 2 pixels.

\subsection{Reconstruction methods}
\label{rec_mthds}
Here we detail our reconstruction methods. We consider filtered backprojection ideas from lambda tomography and algebraic reconstruction methods as detailed below:
\begin{enumerate}
\item Filtered backprojection (FBP) - here, we recover the edges of $f$ via
\begin{equation}
f_r = R^* \frac{\mathrm{d}^2}{\mathrm{d}s^2} (\mathcal{R}f)_\gamma.
\end{equation}
This is a classical idea in lambda tomography \cite{quinto2009electron}. Here, $(\mathcal{R}f)_\gamma$ is the true data perturbed by Gaussian noise as discussed above.
\item Landweber iteration  - Let $A$ be the discretized form of the linear Radon transform, $R$. For this method, we apply standard Landweber iteration \cite{landweber1951iteration}  with $\vb_\gamma$ (i.e., the non-linear data) as our data, and using $A$ as the operator. To implement the Landweber method, we use the code supplied in \cite{AIRtools}.
\item Total Variation (TV) -  here, we find
\begin{equation}
\label{solve}
\argmin_{\vx} \|A\vx - \vb_\gamma\|^2 + \lambda \text{TV}(\vx),
\end{equation}
where $\lambda$ is the smooting parameter, $\text{TV}(\vx) = \sqrt{\|\nabla\vx\|_2^2 + \beta^2}$ is a modified TV norm, and $\beta$ is an additional smoothing parameter that is introduced so that the gradient of $G$ is defined at zero. To implement TV, we use the code of \cite{ehrhardt2014joint}.
\end{enumerate}
In all methods detailed above, we are treating the non-linear data as linear, and applying some inversion based on $R_w$. While we cannot recover accurately the smooth parts of the density in this way, we expect the edges of $f$ to be highlighted in the reconstruction by Corollary \ref{corr_main_2}. To extract the edge map of $f$ from the reconstructed image, we use edge detection methods. Specifically, we apply the Matlab function, ``edge."

\subsection{Image phantoms and physical parameters}
We consider the image phantoms displayed in figure \ref{F1}. The left-hand phantom in figure \ref{F1a} has non-convex boundary and is supported on a simply connected domain and the right-hand phantom in figure \ref{F1b} is an elliptic annulus with two disconnected convex boundaries (one exterior and one interior). 
\begin{figure}
\centering
\begin{subfigure}{0.3\textwidth}
\includegraphics[width=0.9\linewidth, height=7cm, keepaspectratio]{NC_phantom}
\subcaption{non-convex boundary} \label{F1a}
\end{subfigure}
\begin{subfigure}{0.3\textwidth}
\includegraphics[width=0.9\linewidth, height=7cm, keepaspectratio]{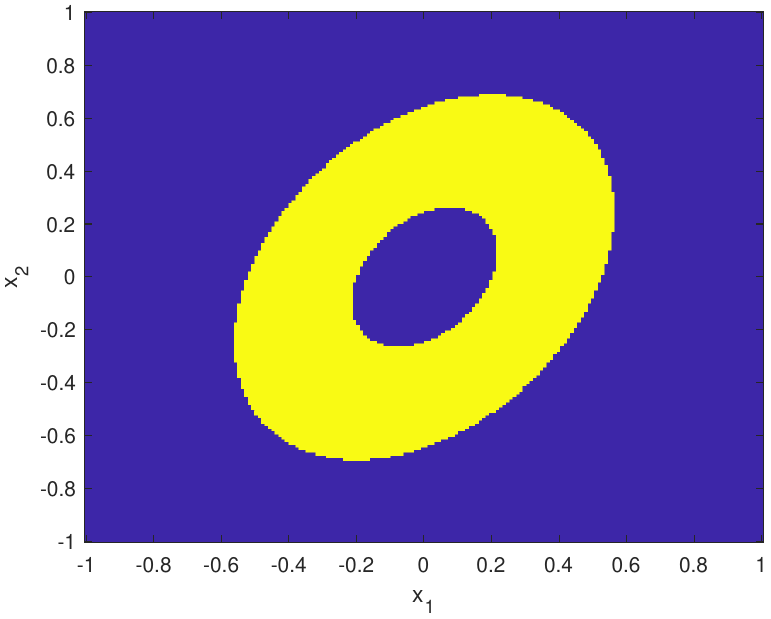}
\subcaption{elliptic annulus} \label{F1b}
\end{subfigure}
\caption{Ground truth phantoms.}
\label{F1}
\end{figure}
For these examples, the phantom material is water with relative electron density 1. One unit of length is 25cm, i.e., the square domain that the phantoms are shown on ($[-1,1]$) is $50\text{cm}^2$, and thus the size of the phantoms is on the order of 10cm (a significant size). We assume the sources are all monocromatic gamma rays with energy $1.17$MeV, which is one of the emitted energies of Co-60, a common gamma ray source used in industrial radiography \cite[section 6.1.1]{national2021radioactive}.

\subsection{Edge recovery results}
Here we present our edge reconstruction results. In figure \ref{F2}, we show edge map reconstruction of the non-convex phantom using the reconstruction methods detailed above. The noise level used throughout this section is $\gamma = 0.01$ (i.e., $1\%$ noise). As expected, the edges of $f$ are highlighted in the FBP reconstruction, although some edges have greater magnitude than others due to attenuation and solid angle effects. In particular, the convex parts of the boundary are generally better highlighted than the non-convex parts as the incoming and outgoing rays are less likely to interact with the material when $L(s,\theta)$ is tangent to a convex part of the boundary. The convex parts of the boundary of the phantom are also typically closer to the source and detector array, and so there is less reduction in signal due to solid angle. This is also due to the specific scanning geometry used here in figure \ref{fig1}.
\begin{figure}
\centering
\begin{subfigure}{0.24\textwidth}
\includegraphics[width=0.9\linewidth, height=3.2cm, keepaspectratio]{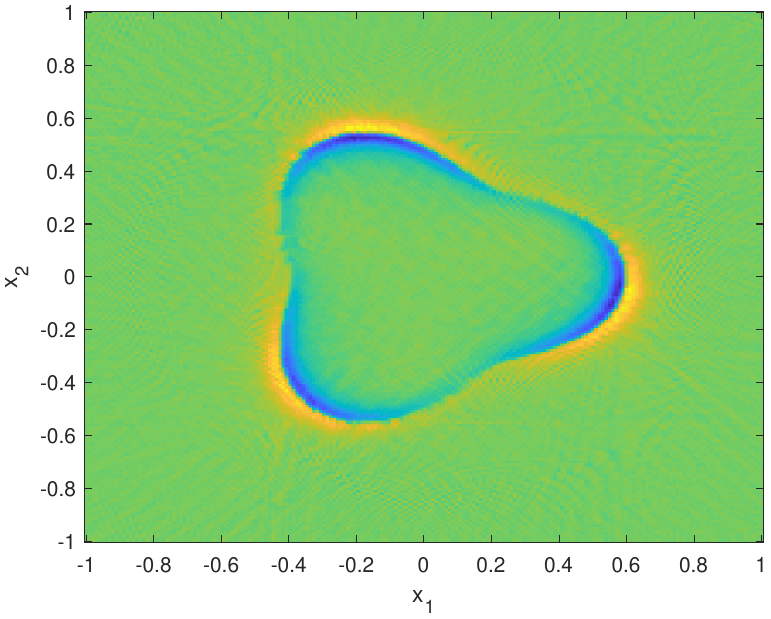}
\end{subfigure}
\begin{subfigure}{0.24\textwidth}
\includegraphics[width=0.9\linewidth, height=3.2cm, keepaspectratio]{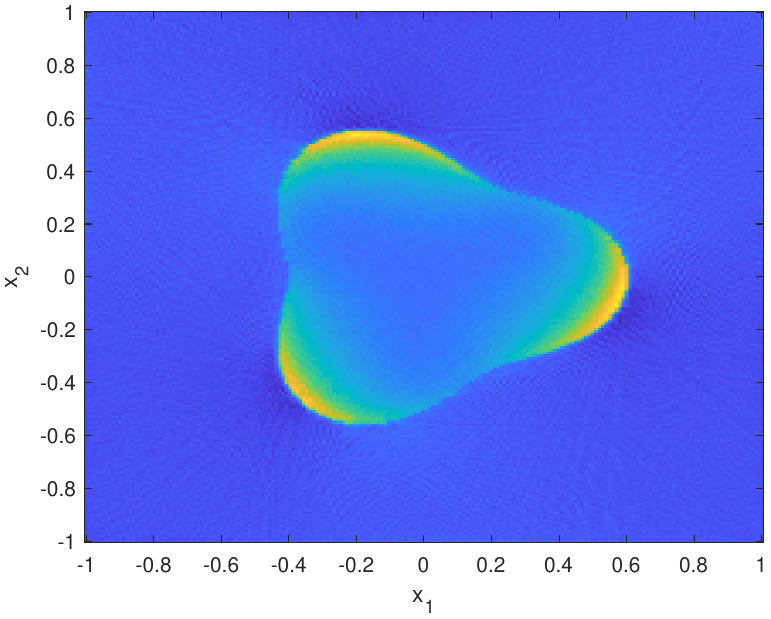}
\end{subfigure}
\begin{subfigure}{0.24\textwidth}
\includegraphics[width=0.9\linewidth, height=3.2cm, keepaspectratio]{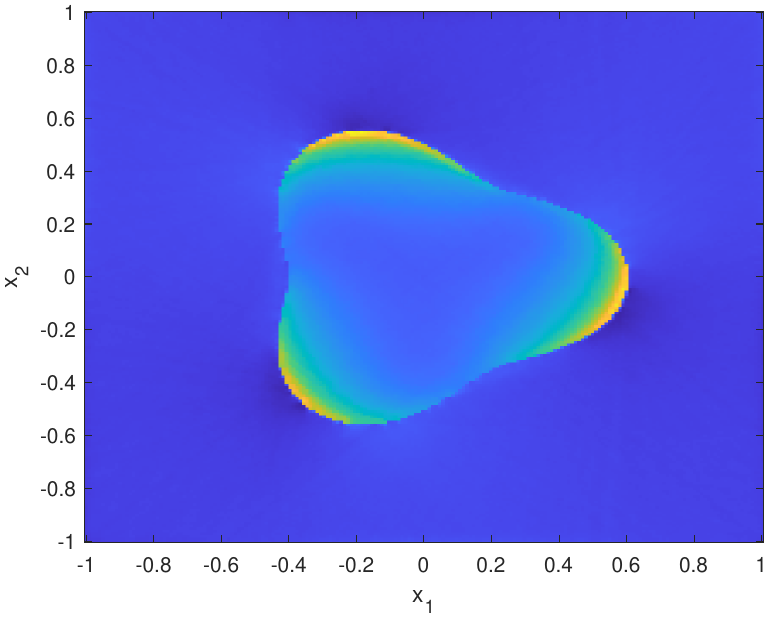}
\end{subfigure}
\\
\begin{subfigure}{0.24\textwidth}
\includegraphics[width=0.9\linewidth, height=3.2cm, keepaspectratio]{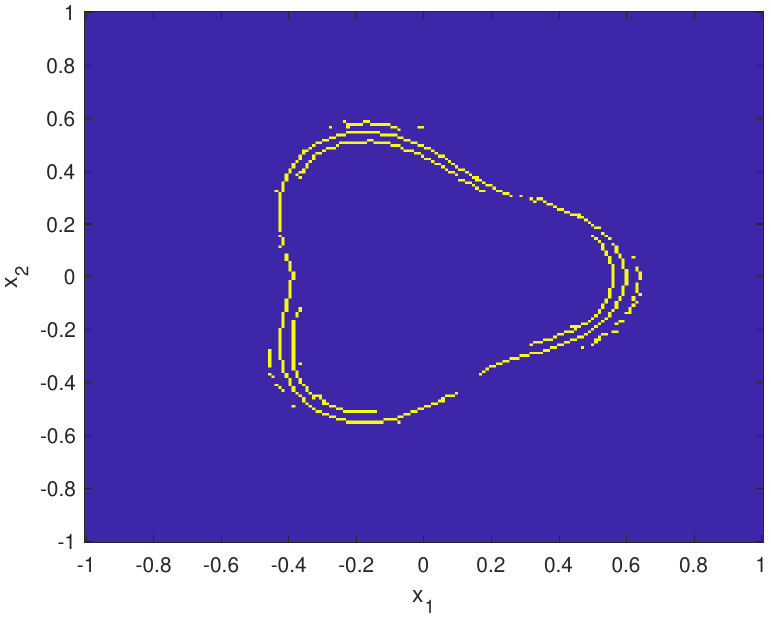}
\subcaption*{FBP}
\end{subfigure}
\begin{subfigure}{0.24\textwidth}
\includegraphics[width=0.9\linewidth, height=3.2cm, keepaspectratio]{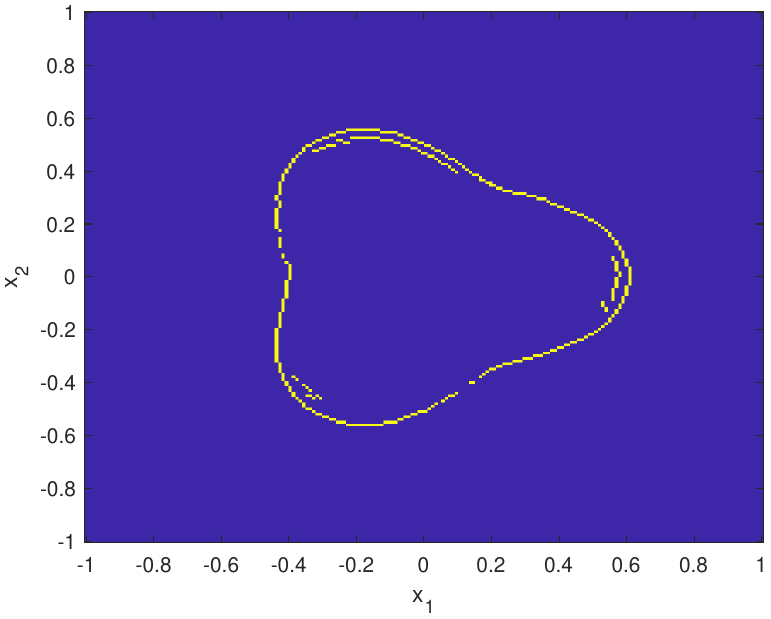}
\subcaption*{Landweber}
\end{subfigure}
\begin{subfigure}{0.24\textwidth}
\includegraphics[width=0.9\linewidth, height=3.2cm, keepaspectratio]{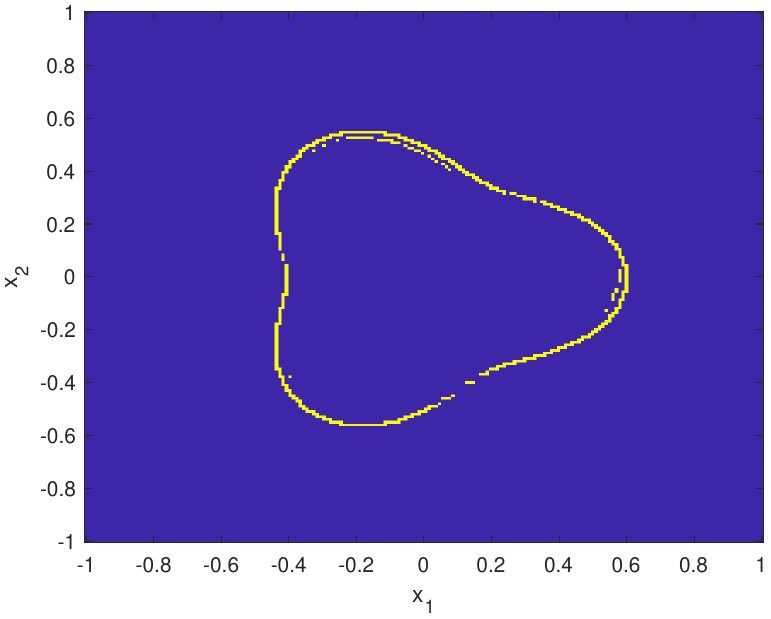}
\subcaption*{TV}
\end{subfigure}
\caption{Non-convex phantom boundary reconstructions with $\gamma = 0.01$. The top row are the reconstructions and the bottom row are the extracted edge maps.}
\label{F2}
\end{figure}
The Landweber and TV reconstructions offer similar results, although the TV method appears better at handling noise which leads to a more accurate realization of the edges. As noted in section \ref{rec_mthds}, these reconstructions do not recover accurately the smooth parts of $f$, and are purely used for edge recovery.

In the non-convex phantom reconstruction, some edges were highlighted more than others due to attenuation and solid angle effects. In the next example, we show such weighting due to physical modeling can lead to more detrimental results in terms of edge recovery. See figure \ref{F3}, where we show edge reconstructions of the ellptic annulus phantom. In this case, as the gamma rays cannot reach the interior boundary of the annulus without attenuating through a significant amount of material, the interior edges are masked and barely visible in the reconstruction. The interior edges are also typically farther away from the sources and detectors by the geometry in figure \ref{fig1}, and so there is less signal contribution from the interior parts of the density due to solid angle. 
\begin{figure}
\centering
\begin{subfigure}{0.24\textwidth}
\includegraphics[width=0.9\linewidth, height=3.2cm, keepaspectratio]{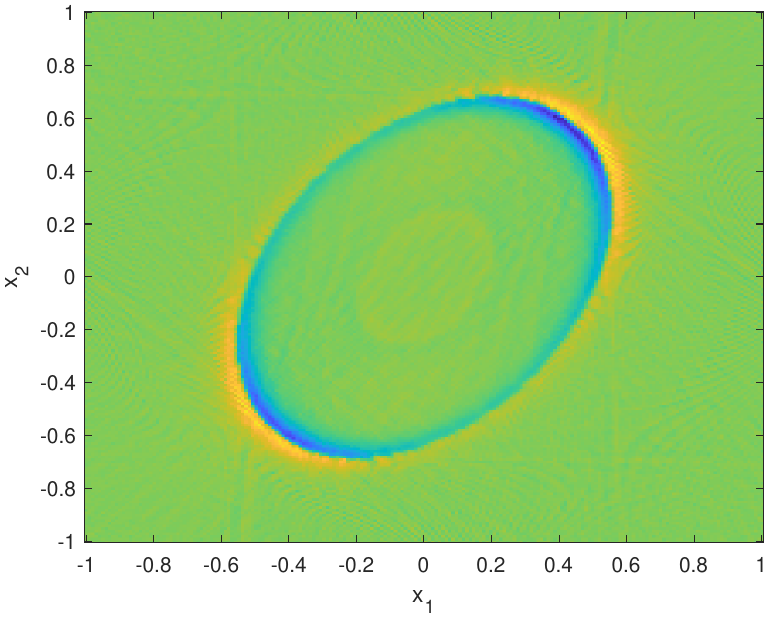}
\end{subfigure}
\begin{subfigure}{0.24\textwidth}
\includegraphics[width=0.9\linewidth, height=3.2cm, keepaspectratio]{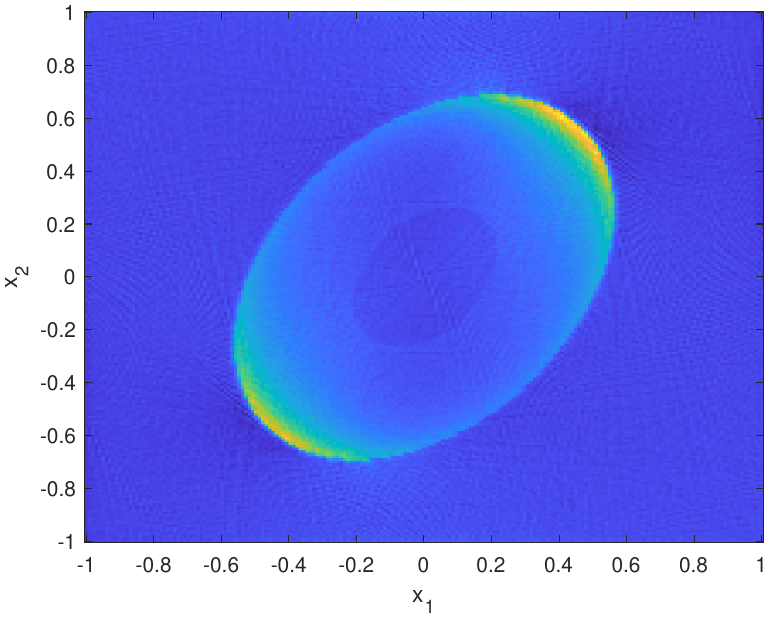}
\end{subfigure}
\begin{subfigure}{0.24\textwidth}
\includegraphics[width=0.9\linewidth, height=3.2cm, keepaspectratio]{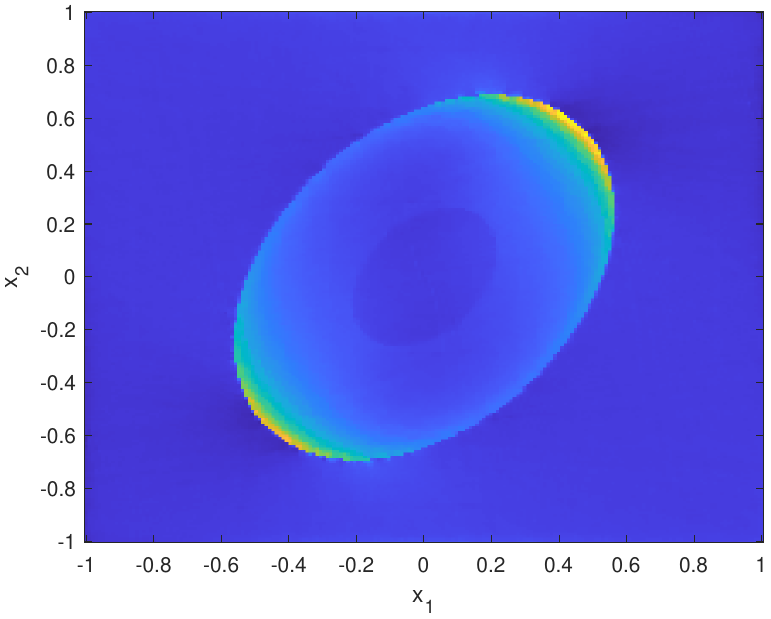}
\end{subfigure}
\\
\begin{subfigure}{0.24\textwidth}
\includegraphics[width=0.9\linewidth, height=3.2cm, keepaspectratio]{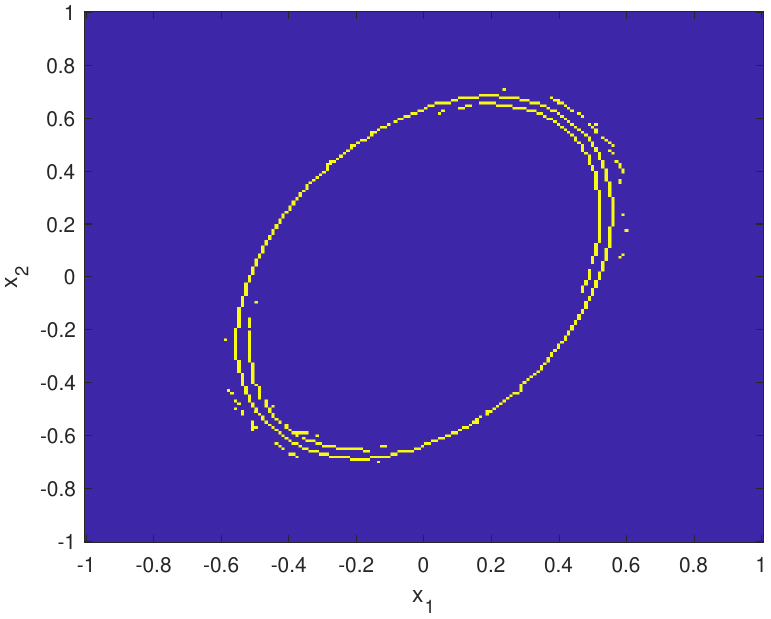}
\subcaption*{FBP}
\end{subfigure}
\begin{subfigure}{0.24\textwidth}
\includegraphics[width=0.9\linewidth, height=3.2cm, keepaspectratio]{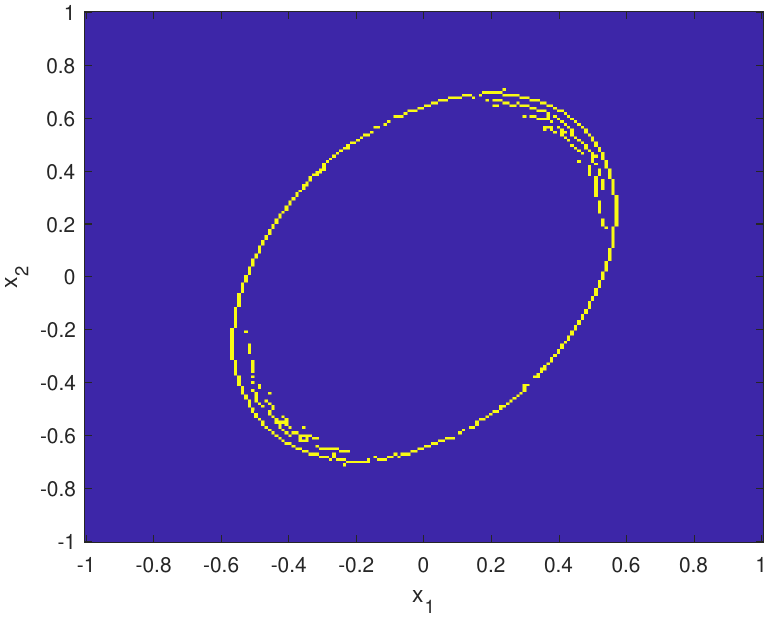}
\subcaption*{Landweber}
\end{subfigure}
\begin{subfigure}{0.24\textwidth}
\includegraphics[width=0.9\linewidth, height=3.2cm, keepaspectratio]{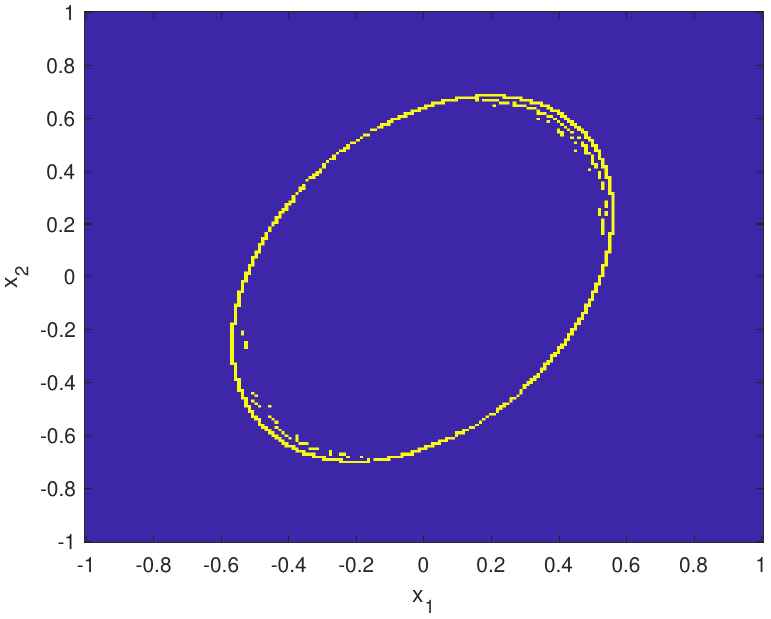}
\subcaption*{TV}
\end{subfigure}
\caption{Elliptic annulus phantom boundary reconstructions with $\gamma = 0.01$. The top row are the reconstructions and the bottom row are the extracted edge maps.}
\label{F3}
\end{figure}
To explain this further, see figure \ref{F4}, where we have plotted sinograms of the non-linear and linear data. In the linear sinograms in figures \ref{F4c} and \ref{F4d}, the parts of the sinogram which correspond to the interior edges of the elliptic annulus are clearly highlighted. The same is not true in the non-linear sinograms in figures \ref{F4a} and \ref{F4b}. By Corollary \ref{corr_main}, the strongest singularities of the non-linear transform and those of the linear transform coincide, at least in terms of Sobolev order. However, Sobolev order is invariant to scaling and thus in practice in this example, the interior edges of the elliptic annulus are not visible due to attenuation. 
\begin{figure}
\centering
\begin{subfigure}{0.24\textwidth}
\includegraphics[width=0.9\linewidth, height=3.2cm, keepaspectratio]{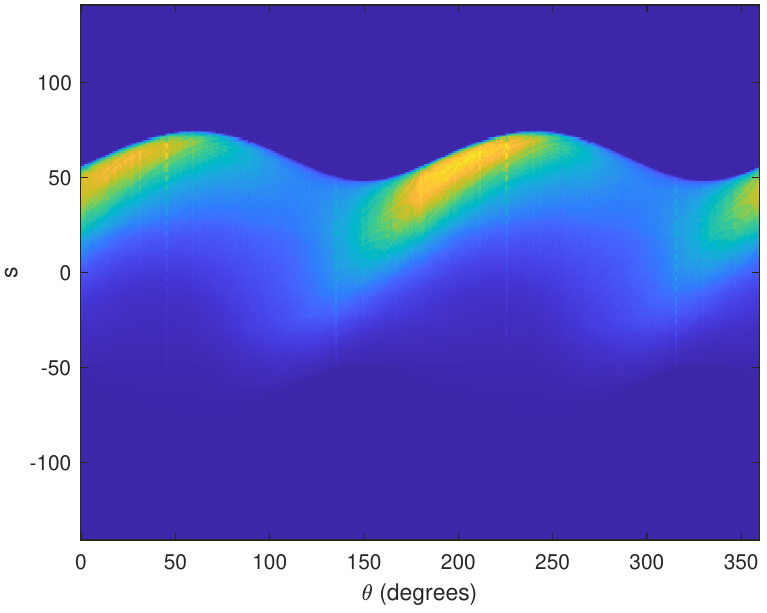}
\subcaption{$\mathcal{R}f$}  \label{F4a}
\end{subfigure}
\begin{subfigure}{0.24\textwidth}
\includegraphics[width=0.9\linewidth, height=3.2cm, keepaspectratio]{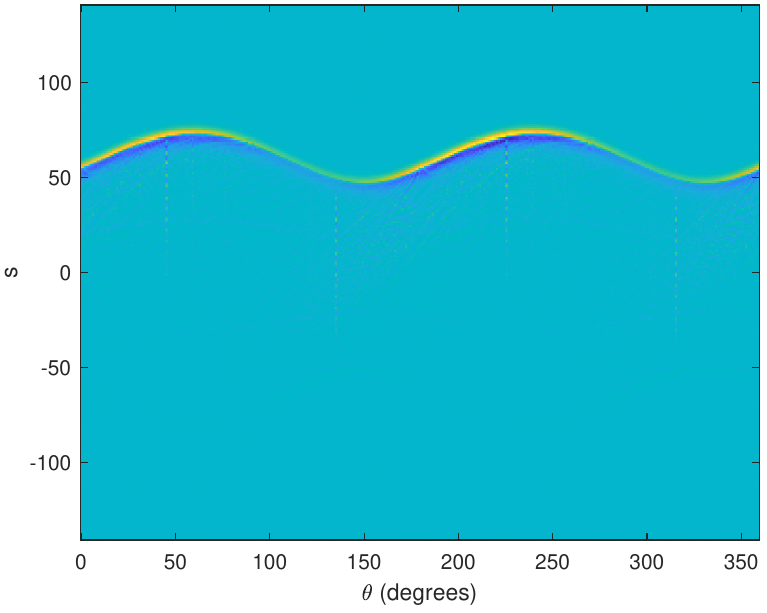}
\subcaption{$\frac{\mathrm{d}^2}{\mathrm{d}s^2}\mathcal{R}f$}  \label{F4b}
\end{subfigure}
\begin{subfigure}{0.24\textwidth}
\includegraphics[width=0.9\linewidth, height=3.2cm, keepaspectratio]{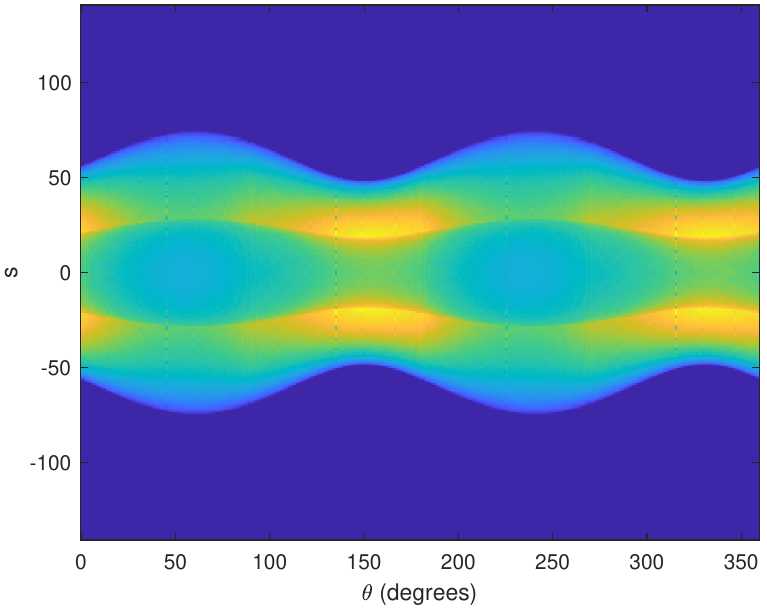}
\subcaption{$R_wf$} \label{F4c}
\end{subfigure}
\begin{subfigure}{0.24\textwidth}
\includegraphics[width=0.9\linewidth, height=3.2cm, keepaspectratio]{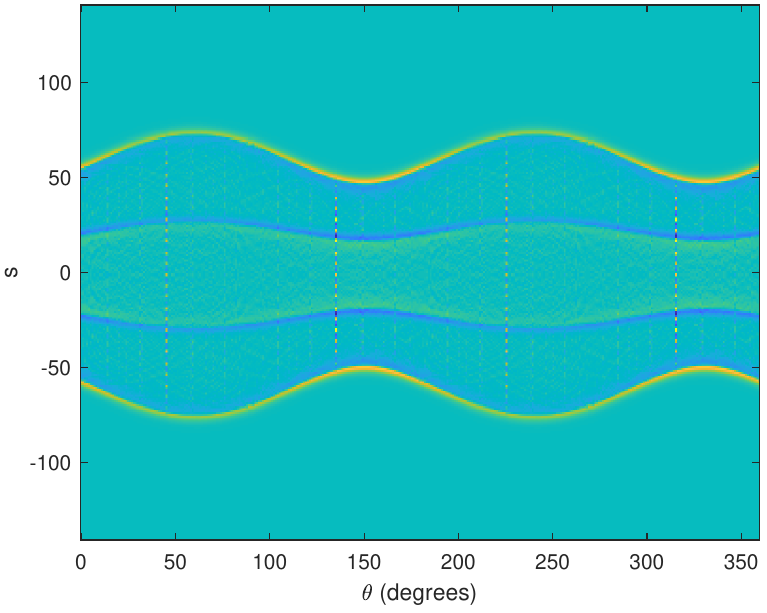}
\subcaption{$\frac{\mathrm{d}^2}{\mathrm{d}s^2}R_wf$}  \label{F4d}
\end{subfigure}
\caption{Linear and non-linear sinograms corresponding to when $f$ is the elliptic annulus phantom. We plot both the sinograms and the second derivatives in $s$ to highlight the singularities.}
\label{F4}
\end{figure}

\subsection{Recovering the density value}
In this subsection, we give an example where we recover the full support of $f$ and the the density value $n_e$ when $f = n_e \chi_\Omega$. Specifically, we consider the case when $f$ is the non-convex phantom and $n_e = 1$.
\begin{figure}
\centering
\begin{subfigure}{0.24\textwidth}
\includegraphics[width=0.9\linewidth, height=3.2cm, keepaspectratio]{NC_phantom}
\subcaption{ground truth \textcolor{white}{Phantom text}}  \label{F5a}
\end{subfigure}
\begin{subfigure}{0.24\textwidth}
\includegraphics[width=0.9\linewidth, height=3.2cm, keepaspectratio]{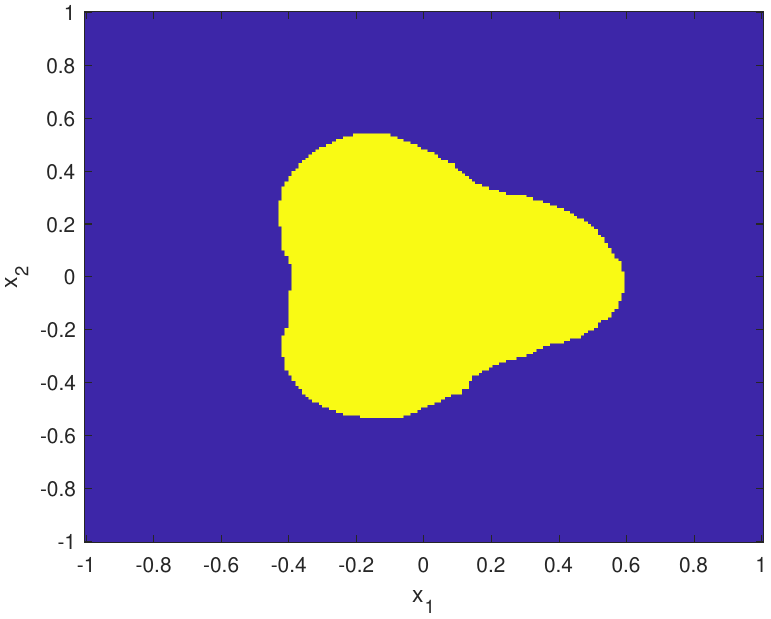}
\subcaption{FBP ($p = 98.9\%$) \textcolor{white}{Phantom text}}  \label{F5b}
\end{subfigure}
\begin{subfigure}{0.24\textwidth}
\includegraphics[width=0.9\linewidth, height=3.2cm, keepaspectratio]{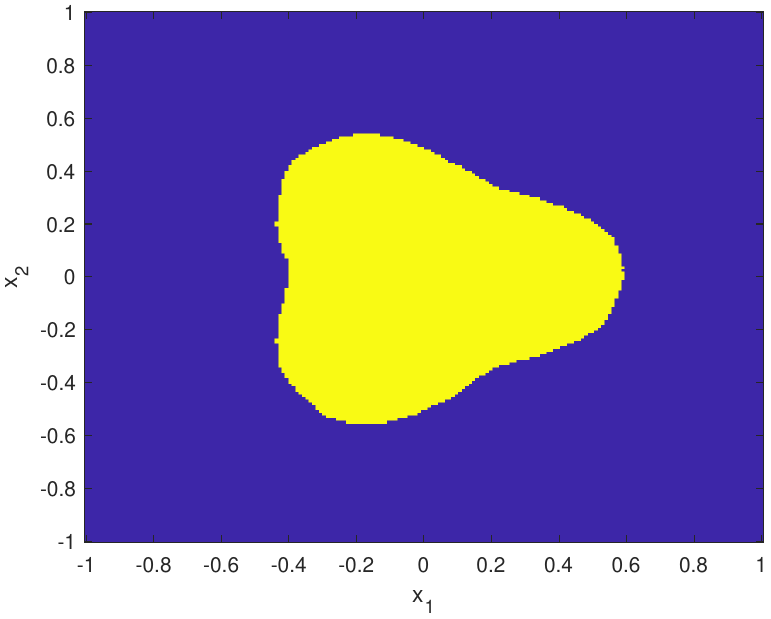}
\subcaption{Landweber ($p = 99.5\%$)} \label{F5c}
\end{subfigure}
\begin{subfigure}{0.24\textwidth}
\includegraphics[width=0.9\linewidth, height=3.2cm, keepaspectratio]{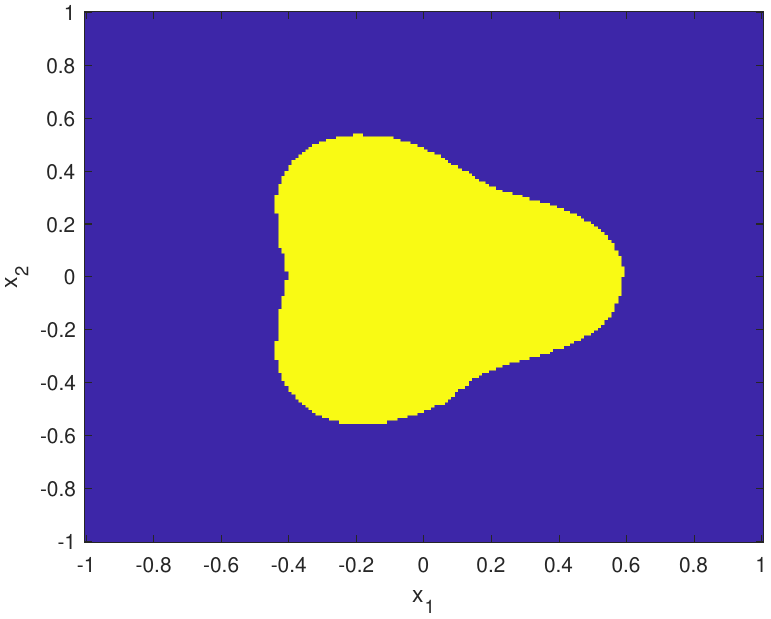}
\subcaption{TV ($p = 99.4\%$) \textcolor{white}{Phantom text}}  \label{F5d}
\end{subfigure}
\caption{Non-convex phantom support reconstructions using all methods considered when $\gamma = 0.01$. The ground truth phantom is shown in the left-hand figure for reference.}
\label{F5}
\end{figure}
See figure \ref{F5}, where we show reconstructions of the support of $f$. To calculate the support, we first approximate the continuous edge using the edge maps on the bottom row of figure \ref{F2}. To do this, we simply take averages over groups of pixels on small neighborhoods of the boundaries in figure \ref{F2}. We then fill in the continuous boundary to calculate the support. We can do this in this example as $\Omega$ is simply connected. In the supcaptions in figure \ref{F5}, we state the percentage of pixels, $p$, which are the same as those in the ground truth, assuming the images are binary. All methods work well in recovering the support, with approximately $99\%$ of $\Omega$ correctly recovered, and the Landweber and TV methods are most optimal.

Now that we have an approximation for $\Omega$, we are in a position to calculate $n_e$. To do this, we determine the $n_e$ which minimizes the least-squares residual, i.e., the $n_e$ which best matches the data given our approximation for $\Omega$.
\begin{figure}
\centering
\begin{subfigure}{0.4\textwidth}
\includegraphics[width=0.9\linewidth, height=7cm, keepaspectratio]{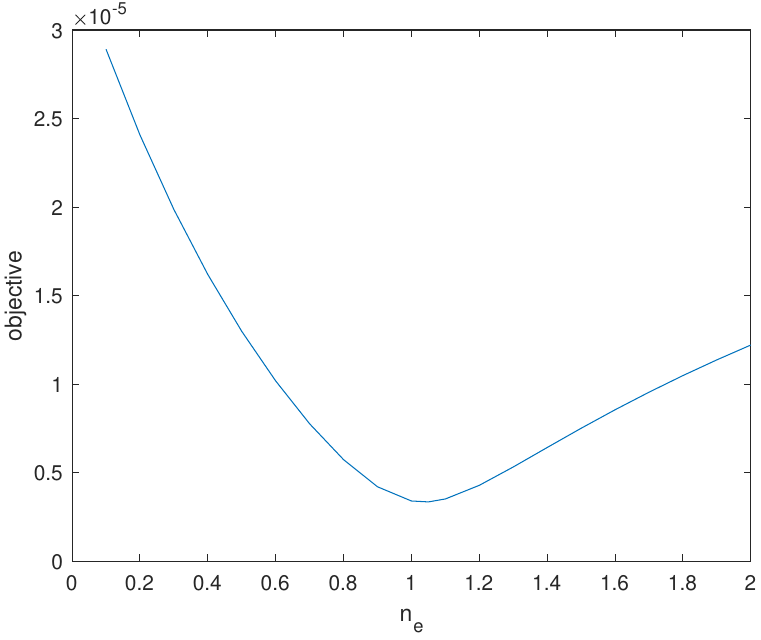}
\end{subfigure}
\caption{Plot of least squares residual for varying $n_e$, using the approximation for $\Omega$ in figure \ref{F5d} (the TV reconstruction).}
\label{F6}
\end{figure}
See figure \ref{F6}, where we plot the residual for varying $n_e$ using our approximation for $\Omega$ in figure \ref{F5d}. The residual is minimized when $n_e \approx 1.05$, which is close to the true value of $n_e = 1$ ($5\%$ error). The curve shown in figure \ref{F6} is in line with Theorem \ref{ne_thm}, as there appears to be a unique minimum when $n_e \in [0,2]$.

\section{Conclusion}
We have presented a microlocal analysis of a novel non-linear ray transform, $\mathcal{R}$, arising in CST. We derived new results on distributional products and the smoothing properties of the V-line transform, and used these to analyze the singularities of $\mathcal{R}f$. Specifically, the singularities of $\mathcal{R}f$ that are not locally in $H^1$ correspond to those of $R_wf$, i.e., the linear analog of $\mathcal{R}f$. We used this result to recover the singularities of $f$ and derive injectivity results for $\mathcal{R}$ when $f = u \chi_\Omega$, where $u > 0$ is smooth and $\Omega$ is a simply connected domain. Simulated experiments were presented to validate the theory and we provided edge map reconstructions of a non-convex and elliptic annulus phantom. In the non-convex phantom example, the edges were recovered accurately, and we used this to fully recover the support of $f$ and the density value. On the elliptic annulus phantom, where the support is not simply connected, we found the attenuation masked the internal boundary as the jump size was insignificant compared to the external boundary. The jump size does not affect Sobolev scale, and thus this does not contradict our theory. However, this is an important practical consideration. We note also that the elliptic annulus was deliberately chosen to be thick (10-15cm) and thus strongly attenuating, and this can be considered a special case when the method is less optimal.

The injectivity results presented here require smoothing kernel $\varphi \in H^{\lambda,\infty}(\mathbb{R}^2)$, where $\lambda > 1$. In future work, we aim to relax this assumption, e.g., so that $\varphi = \delta$ (which would allow a point-source model). We also wish to strengthen the results of Theorem \ref{mainprod}, in particular part (2). We hypothesize that (2) is provable when $q = p = r = 2$, which, e.g., would reduce the smoothness requirements on $g$ in this application. 

\section*{Acknowledgements:} 
%We would like to thank Marika Horsky for translating the abstract from English into French.   
The first author wishes to acknowledge funding support from Aspira Women's Health, The Cleveland Clinic Foundation, The Honorable Tina Brozman Foundation, the V Foundation, and the National Cancer Institute R03CA283252-01. Sean Holman was supported by Engineering and Physical Sciences Research Council (EPSRC) through grant EP/V007742/1. We also wish to thank The Institute for Computational and Experimental Research in Mathematics (ICERM) at Brown University for hosting us for a workshop in August 2024, which facilitated compelling discussions and advancements in this project.

\bibliographystyle{abbrv} 
\bibliography{RefRevolution}

\begin{thebibliography}{10}

\bibitem{f2}
L.~Brandolini.
\newblock Fourier transform of characteristic functions and lebesgue constants for multiple fourier series.
\newblock In {\em Colloquium Mathematicae}, volume~65, pages 51--59, 1993.

\bibitem{sobolev}
H.~Brezis, P.~Mironescu, et~al.
\newblock Composition in fractional sobolev spaces.
\newblock {\em Discrete and Continuous Dynamical Systems}, 7(2):241--246, 2001.

\bibitem{Co1963}
A.~M. Cormack.
\newblock {Representation of a function by its line integrals with some radiological applications}.
\newblock {\em J. Appl. Physics}, 34(9):2722--2727, 1963.

\bibitem{duistermaat1996fourier}
J.~J. Duistermaat and L.~Hormander.
\newblock {\em {F}ourier integral operators}, volume~2.
\newblock Springer, 1996.

\bibitem{ehrhardt2014joint}
M.~J. Ehrhardt, K.~Thielemans, L.~Pizarro, D.~Atkinson, S.~Ourselin, B.~F. Hutton, and S.~R. Arridge.
\newblock Joint reconstruction of pet-mri by exploiting structural similarity.
\newblock {\em Inverse Problems}, 31(1):015001, 2014.

\bibitem{farmer1971new}
F.~Farmer and M.~P. Collins.
\newblock A new approach to the determination of anatomical cross-sections of the body by compton scattering of gamma-rays.
\newblock {\em Physics in Medicine \& Biology}, 16(4):577, 1971.

\bibitem{farmer1974further}
F.~Farmer and M.~P. Collins.
\newblock A further appraisal of the compton scattering method for determining anatomical cross-sections of the body.
\newblock {\em Physics in Medicine \& Biology}, 19(6):808, 1974.

\bibitem{hamaker1980divergent}
C.~Hamaker, K.~Smith, D.~Solmon, and S.~Wagner.
\newblock The divergent beam x-ray transform.
\newblock {\em The Rocky Mountain Journal of Mathematics}, 10(1):253--283, 1980.

\bibitem{AIRtools}
P.~C. Hansen and J.~S. J\o~rgensen.
\newblock A{IR} {T}ools {II}: algebraic iterative reconstruction methods, improved implementation.
\newblock {\em Numer. Algorithms}, 79(1):107--137, 2018.

\bibitem{holt1985compton}
R.~Holt.
\newblock Compton imaging.
\newblock {\em Endeavour}, 9(2):97--105, 1985.

\bibitem{hormanderI}
L.~H{\"o}rmander.
\newblock {\em The analysis of linear partial differential operators. {I}}.
\newblock Classics in Mathematics. Springer-Verlag, Berlin, 2003.
\newblock Distribution theory and {F}ourier analysis, Reprint of the second (1990) edition [Springer, Berlin].

\bibitem{hormanderIII}
L.~H\"{o}rmander.
\newblock {\em The analysis of linear partial differential operators. {III}}.
\newblock Classics in Mathematics. Springer, Berlin, 2007.
\newblock Pseudo-differential operators, Reprint of the 1994 edition.

\bibitem{hormander}
L.~H\"{o}rmander.
\newblock {\em The analysis of linear partial differential operators. {IV}}.
\newblock Classics in Mathematics. Springer-Verlag, Berlin, 2009.
\newblock Fourier integral operators, Reprint of the 1994 edition.

\bibitem{klein1929streuung}
O.~Klein and Y.~Nishina.
\newblock {\"U}ber die streuung von strahlung durch freie elektronen nach der neuen relativistischen quantendynamik von dirac.
\newblock {\em Zeitschrift f{\"u}r Physik}, 52(11):853--868, 1929.

\bibitem{landweber1951iteration}
L.~Landweber.
\newblock An iteration formula for fredholm integral equations of the first kind.
\newblock {\em American journal of mathematics}, 73(3):615--624, 1951.

\bibitem{lee2012smooth}
J.~M. Lee.
\newblock {\em Introduction to Smooth Manifolds, Second Edition}.
\newblock Number 218 in Graduate Texts in Mathematics. Springer, 2013.

\bibitem{mishra2025tensor}
R.~K. Mishra, A.~Purohit, and I.~Zamindar.
\newblock Tensor tomography using v-line transforms with vertices restricted to a circle.
\newblock {\em Analysis and Mathematical Physics}, 15(1):23, 2025.

\bibitem{national2021radioactive}
{National Academies of Sciences, Engineering, and Medicine and others}.
\newblock Radioactive sources: applications and alternative technologies.
\newblock 2021.

\bibitem{natterer}
F.~Natterer.
\newblock {\em {The mathematics of computerized tomography}}.
\newblock Classics in Mathematics. Society for Industrial and Applied Mathematics (SIAM), New York, 2001.

\bibitem{NT}
M.~Nguyen and T.~T. Truong.
\newblock {Inversion of a new circular-arc Radon transform for {C}ompton scattering tomography}.
\newblock {\em Inverse Problems}, 26(6):065005, 2010.

\bibitem{norton}
S.~J. Norton.
\newblock {Compton scattering tomography}.
\newblock {\em Journal of applied physics}, 76(4):2007--2015, 1994.

\bibitem{pal}
V.~P. Palamodov.
\newblock {An analytic reconstruction for the {C}ompton scattering tomography in a plane}.
\newblock {\em Inverse Problems}, 27(12):125004, 2011.

\bibitem{quinto1988tomographic}
E.~T. Quinto.
\newblock Tomographic reconstructions from incomplete data-numerical inversion of the exterior radon transform.
\newblock {\em Inverse Problems}, 4(3):867, 1988.

\bibitem{Q1993sing}
E.~T. Quinto.
\newblock {Singularities of the {X}-ray transform and limited data tomography in ${\mathbb R}^2$ and ${\mathbb R}^3$}.
\newblock {\em SIAM J. Math. Anal.}, 24:1215--1225, 1993.

\bibitem{quinto2009electron}
E.~T. Quinto, U.~Skoglund, and O.~{\"O}ktem.
\newblock Electron lambda-tomography.
\newblock {\em Proceedings of the National Academy of Sciences}, 106(51):21842--21847, 2009.

\bibitem{rigaud20213d}
G.~Rigaud.
\newblock 3d compton scattering imaging with multiple scattering: Analysis by fio and contour reconstruction.
\newblock {\em Inverse Problems}, 2021.

\bibitem{rigaud2021reconstruction}
G.~Rigaud and B.~Hahn.
\newblock Reconstruction algorithm for 3d compton scattering imaging with incomplete data.
\newblock {\em Inverse Problems in Science and Engineering}, 29(7):967--989, 2021.

\bibitem{rigaud20183d}
G.~Rigaud and B.~N. Hahn.
\newblock 3{D} {C}ompton scattering imaging and contour reconstruction for a class of {R}adon transforms.
\newblock {\em Inverse Problems}, 34(7):075004, 2018.

\bibitem{Rudin:FA}
W.~Rudin.
\newblock {\em Functional analysis}.
\newblock McGraw-Hill Book Co., New York, 1973.
\newblock McGraw-Hill Series in Higher Mathematics.

\bibitem{strecker1982scatter}
H.~Strecker.
\newblock Scatter imaging of aluminium castings using an x-ray fan beam and a pinhole camera.
\newblock {\em Materials Evaluation}, 40(10):1050--1056, 1982.

\bibitem{f1}
I.~Svensson.
\newblock Estimates for the fourier transform of the characteristic function of a convex set.
\newblock {\em Arkiv f{\"o}r Matematik}, 9(1):11--22, 1971.

\bibitem{truong2012recent}
T.~Truong and M.~K. Nguyen.
\newblock Recent developments on compton scatter tomography: theory and numerical simulations.
\newblock {\em Numerical Simulation-From Theory to Industry}, pages 101--128, 2012.

\bibitem{truong2019compton}
T.-T. Truong and M.~K. Nguyen.
\newblock Compton scatter tomography in annular domains.
\newblock {\em Inverse Problems}, 2019.

\bibitem{webberholman}
J.~W. Webber and S.~Holman.
\newblock Microlocal analysis of a spindle transform.
\newblock {\em Inverse Problems \& Imaging}, 13(2):231--261, 2019.

\bibitem{me2}
J.~W. Webber and W.~R. Lionheart.
\newblock {Three dimensional {C}ompton scattering tomography}.
\newblock {\em Inverse Problems}, 34(8):084001, 2018.

\bibitem{WebberQuinto2020II}
J.~W. Webber, E.~T. Quinto, and E.~L. Miller.
\newblock A joint reconstruction and lambda tomography regularization technique for energy-resolved x-ray imaging.
\newblock {\em Inverse Problems}, 36(7):074002, 2020.

\bibitem{ziou2021scale}
D.~Ziou, N.~Nacereddine, and A.~B. Goumeidane.
\newblock Scale space radon transform.
\newblock {\em IET Image Processing}, 15(9):2097--2111, 2021.

\end{thebibliography}

\appendix
\section{Addtional analysis of the V-line transform}
\label{add_vline}
We have the theorem which calculates the Fourier components of $\mathcal{V}_{a,b}f$.
\begin{theorem}
The weighted V-line transform has the Fourier decomposition:
\begin{equation}
\widehat{ \mathcal{V}_{a,b} f}_k(\xi) = \int_{-\pi}^\pi  \mathcal{F}\paren{\mathcal{V}_{a,b} f}(\xi, \phi) e^{-i k \phi} \mathrm{d}\phi = \paren{a + be^{2 i k\psi}}e^{-ik(\varphi_{\xi} + \frac{\pi}{2})} \frac{\hat{f}(\xi)}{|\xi|},
\end{equation}
where $2\psi$ is the angle between the V-lines and $\xi = |\xi|(\cos\varphi_{\xi}, \sin\varphi_{\xi})$.
\begin{proof}
We have
\begin{equation}
\label{fourier}
\begin{split}
\widehat{ \mathcal{V}_{a,b} f}_k(\xi) &= \frac{1}{2\pi}\int_{\mathbb{R}^2}\int_{-\pi}^{\pi}\mathcal{V}_{a,b}f(\vy,\phi)e^{-i k \phi} e^{-i \xi \cdot \vy} \mathrm{d}\phi \mathrm{d}\vy\\
&= \hat{f}(\xi)\frac{1}{2\pi} \left[a\int_0^\infty \int_{-\pi}^{\pi}e^{-i\paren{ k\phi + t\Phi\cdot\xi }}\mathrm{d}\phi\mathrm{d}t + b\int_0^\infty \int_{-\pi}^{\pi}e^{-i\paren{ k\phi + t\Phi'\cdot\xi }}\mathrm{d}\phi\mathrm{d}t \right]\\
& = \paren{a+be^{2ik\psi}}e^{-ik(\varphi_{\xi} + \frac{\pi}{2})} \left[\int_0^\infty J_k(t|\xi|)\mathrm{d}t\right] \hat{f}(\xi)\\
& = \paren{a+be^{2ik\psi}}e^{-ik(\varphi_{\xi} + \frac{\pi}{2})} \frac{\hat{f}(\xi)}{|\xi|},
\end{split}
\end{equation}
where  $\xi = |\xi|(\cos\varphi_{\xi},\sin\varphi_{\xi})$, where $\varphi_{\xi} \in [-\pi,\pi]$, and $J_k$ is a Bessel function order $k$. The last step follows since
$$\mathcal{L}(J_k(t|\xi|)) = \int_0^\infty e^{-st}J_k(t|\xi|)\mathrm{d}t = \frac{\paren{\sqrt{s^2+|\xi|^2} - s}^k}{|\xi|^k\sqrt{s^2+|\xi|^2}},$$
for $\text{Re}\ s > 0$, where $\mathcal{L}$ is the Laplace transform. Letting $s\to 0$ shows $\int_0^\infty J_k(t|\xi|)\mathrm{d}t = \frac{1}{|\xi|}$.
\end{proof}
\end{theorem}
An interesting corollary of this is $\mathcal{V}_{a,b}f_0(\vy) = (a+b) f \ast \paren{\frac{1}{|\vy|}} = (a+b)R_h^*R_hf(\vy),$ where $R_h$ is the classical hyperplane Radon transform in 2-D. We do not need this result for our analysis but this may be of interest to other researchers who study the V-line transform as it provides an inverse for $\mathcal{V}_{a,b}$. 

The above theorem implies that the fourier components of $\mathcal{V}_{a,b} f$ decay faster  (due to the division by $|\xi|$), at least in the $\xi$ variable which is dual to $\vx$. However, the series components in the $\phi$ direction do not decay at all, which makes it difficult to determine global smoothing estimates. Theorem \ref{V_smooth} in the main text rectifies this somewhat, when the length of the V-line rays is finite, as in this case, the V-line transform is compactly supported in the $\vx$ variable.

\end{document}